\documentclass[letter,11pt]{amsart}

\usepackage{amsmath,amsthm,amssymb,amsfonts,enumerate,color,upgreek,float}
\usepackage[pdftex]{graphicx}
\usepackage{tikz}
\usepackage[margin=1in,top=1in]{geometry}
\usepackage[latin1]{inputenc}
\usepackage{hyperref}
\usepackage{enumitem}

\newcommand{\R}{\mathbb{R}}
\newcommand{\Rn}{\mathbb{R}^n}
\newcommand{\N}{\mathbb{N}}
\newcommand{\vep}{\varepsilon}
\newcommand{\dist}{\operatorname{dist}}
\newcommand{\mres}{\mathbin{\vrule height 1.6ex depth 0pt width
0.13ex\vrule height 0.13ex depth 0pt width 1.3ex}}
\newcommand{\dhr}{dH^{n-1}\mres \Gamma}
\newcommand{\dini}{\operatorname{Dini}}
\newcommand{\op}{\Omega^+}
\newcommand{\om}{\Omega^-}
\newcommand{\opm}{\Omega^\pm}
\newcommand{\mes}[1]{\, dH^{n-1}\mres #1 }
\newcommand{\gdiniz}{\omega_{g,0}}
\newcommand{\gradiniz}{\omega_{\nabla \psi, 0}}

\newtheorem{thm}{Theorem}[section]
\newtheorem{prop}[thm]{Proposition}
\newtheorem{cor}[thm]{Corollary}
\newtheorem{lem}[thm]{Lemma}
\theoremstyle{definition}
\newtheorem{defn}[thm]{Definition}
\newtheorem{rem}[thm]{Remark}

\allowdisplaybreaks
\numberwithin{equation}{section}

\author[I. U. Erneta]{I\~{n}igo U. Erneta}
    \address{I\~{n}igo U. Erneta. Department of Mathematics\\
    Rutgers University\\
110 Frelinghuysen Rd., Piscataway, NJ 08854, USA}
    \email{inigo.erneta@rutgers.edu}
 
\author[M.~Soria-Carro]{Mar\'ia Soria-Carro}
    \address{Mar\'ia Soria-Carro. Department of Mathematics\\
    Rutgers University\\
110 Frelinghuysen Rd., Piscataway, NJ 08854, USA}
    \email{maria.soriacarro@rutgers.edu}

    \keywords{Poisson's equation, surface measures, transmission problems, regularity, Dini condition.}
\subjclass[2020]{Primary: 35B65, 35J25. Secondary: 35J05.}
\thanks{The second author is partially supported by NSF grant DMS-2247096.}

\begin{document}

\title[Lipschitz regularity for Poisson equations involving measures]{Lipschitz regularity for Poisson equations involving measures supported on $C^{1,\dini}$ interfaces}

\begin{abstract}
We prove optimal Lipschitz regularity of solutions to Poisson's equation with measure data supported on a $C^{1,\dini}$ interface and with $C^{0,\dini}$ density.
We achieve this by deriving pointwise gradient estimates on the interface, further showing the piecewise differentiability of solutions up to this surface.
Our approach relies on perturbation arguments and estimates for the Green's function of the Laplacian.
Additionally, we provide sharp counterexamples highlighting the minimality of our assumptions.
\end{abstract}

\maketitle

\section{Introduction} 

Let $\Omega \subset \R^n$, with $n\geq 3$, be a smooth bounded domain, and let $\Gamma = \partial \op$ for some domain $\op$ compactly contained in $\Omega$.
In this paper, we consider the following 
Dirichlet problem for Poisson's equation involving surface measures
\begin{equation}\label{eq:main}
\begin{cases} 
\Delta u = g \mes{\Gamma} & \hbox{in}~\Omega,\\
u=0 & \hbox{on}~\partial \Omega,
\end{cases}
\end{equation}
where $g$ is a measurable function on the surface~$\Gamma$ and $\mes{\Gamma}$ denotes the Hausdorff measure restricted to~$\Gamma$.
This model describes diffusion processes with jump discontinuities of the gradient across a fixed boundary $\Gamma$, and has applications in physics, biology, and engineering.
The equation in~\eqref{eq:main} has also appeared in the context of free boundary problems, such as 
the Bernoulli one-phase and two-phase problems~\cite{AC, ACF} and the inverse conductivity problem~\cite{ACS}, where $\Gamma = \partial \{u > 0\}$ is 
another unknown of the problem.

Poisson's equation with measure data has been extensively studied since the 20th century,
and its general theory is well-established using techniques from potential theory and PDEs. 
We refer the reader to~\cite{Ponce} for a complete exposition.
For bounded Radon measures $\mu$, the notion of $L^1$-weak solution to $\Delta u =\mu$ was introduced by Littman, Stampacchia, and  Weinberger in the seminal work~\cite{LSW}.
There, the authors showed via a duality method 
that the unique weak solution to the homogeneous Dirichlet problem belongs to the Sobolev space $W_0^{1,p}(\Omega)$, for any $p\in [1, \tfrac{n}{n-1}\big)$. 
When $\mu$ is a surface measure of the form
\begin{equation} \label{eq:measure}
\mu = g \mes{\Gamma},\text{ with } \Gamma~\hbox{Lipschitz and } g \in L^{\infty}(\Gamma),
\end{equation}
then a similar argument improves the regularity of the solution to $W_0^{1,p}(\Omega)$, for all $p\in [1,\infty)$ (see~\cite[Proposition 3.4]{Muller1}).  
By Sobolev embedding, it follows that $u\in C_{\rm loc}^{0,\alpha}(\Omega)$, for all $\alpha\in(0,1)$.
The case $\alpha=1$ is critical for the duality method and thus requires special attention.

The issue of regularity has long been a cornerstone of research in elliptic PDEs, with significant advances in recent decades.
In the context of linear and quasilinear equations involving measure data, several works~\cite{DM,K1,KM,Lieberman,Mingione,RZ} have established continuity and differentiability results under growth assumptions on the measure.
To prove gradient estimates for Poisson's equation $\Delta u = \mu$, classical potential theory leads naturally to the following localized Riesz potential 
\begin{equation*}
I_1^{|\mu|} (x,R) := \int_0^R \frac{|\mu|(B_r(x))}{r^{n-1}} \frac{dr}{r}, \qquad R>0, \quad x\in \Rn,
\end{equation*}
where $|\mu|$ denotes the total variation of $\mu$ (see~\cite{Mingione}). 
In particular, the local \textit{Lipschitz criterion} asserts that if $I_1^{|\mu|} (\cdot,R)\in L^\infty(\Omega)$  for all $R>0$, then $u \in C_{\rm loc}^{0,1}(\Omega)$, this being the optimal regularity. 
Note that if $\mu$ satisfies~\eqref{eq:measure}, 
then $|\mu|(B_r(x)) \sim r^{n-1}$ for $x \in \Gamma$, and $I_1^{|\mu|}(\cdot, R)$ 
blows up on $\Gamma$.
Therefore, surface measures are a borderline case for the above criterion.
Clearly, if $u$ is a weak solution to~\eqref{eq:main}, then $u$ is smooth in $\Omega\setminus \Gamma$ as $u$ is harmonic outside of $\Gamma$.
The main difficulty is to obtain the optimal regularity of $u$ on this surface. 

Problem~\eqref{eq:main} can be formulated as a boundary-value problem describing the behavior of solutions on the surface $\Gamma$.
Indeed, let $\om=\Omega \setminus \overline \op$ and note that $\Gamma=\partial \Omega^+$ is the interface separating the two domains $\Omega^+$ and $\Omega^-$. 
Given a continuous function $u\colon \overline \Omega \to \R$, we write 
\begin{equation*} 
u^+ = u \chi_{\overline\op} \qquad \hbox{and} \qquad  u^- =  u \chi_{\overline\om}.    
\end{equation*}
Formally, weak solutions to~\eqref{eq:main} (see Definition~\ref{def:solution} below) satisfy the transmission problem
\begin{equation} \label{eq:transmission}
    \begin{cases}
    \Delta u = 0 & \hbox{in}~\Omega\setminus\Gamma,\\
    u_{\nu}^{+} - u_{\nu}^{-} = g & \hbox{on}~\Gamma,\\
    u = 0 & \hbox{on}~\partial\Omega,
    \end{cases}
\end{equation}
where $\nu$ is the inward normal vector to $\Omega^+$ and $u_\nu^\pm$ denote the normal derivatives of $u^\pm$.

When $\Gamma$ is a $C^{1,\alpha}$ interface and $g\in C^{0,\alpha}(\Gamma)$ for some $\alpha\in (0,1)$,
Caffarelli, Stinga, and the second author~\cite{CSCS} established that $u^\pm \in C^{1,\alpha}(\overline\Omega^\pm)$.
As a consequence, $u$ is Lipschitz in $\Omega$.
Here, the optimal regularity of $u$ is deduced from the boundary regularity up to $\Gamma$ of the functions $u^\pm$,
which is implied by the transmission condition, i.e., 
the second equation in~\eqref{eq:transmission}. 
Heuristically, problem~\eqref{eq:transmission} can be understood as two separate Neumann problems in $\Omega^+$ and $\Omega^-$
coupled through the transmission condition on the interface $\Gamma$, where $u^+ = u^-$. 
The expected boundary regularity of $u^\pm$ on $\Gamma$ is that of the Neumann problem, which is well studied (see Remark~\ref{rem:counterexample}).
Under the same assumptions, the Lipschitz continuity of solutions to~\eqref{eq:main} has also been obtained by M\"{u}ller~\cite{Muller1, Muller2}
 using different methods. The author pointed out in \cite{Muller2} that this result fails when $\Gamma$ is only Lipschitz.

Our main goal is to find minimal regularity conditions on the data $\Gamma$ and $g$ guaranteeing the Lipschitz continuity of solutions to~\eqref{eq:main}.
The counterexamples in Section~\ref{sec:counterexamples} show
that if $\Gamma$ is only a $C^1$ interface or $g$ is only a continuous function on $\Gamma$, then weak solutions need not be Lipschitz.
To obtain this regularity, a more precise control on the modulus of continuity $\omega$ of $\Gamma$ and $g$ is needed.
We will see that the natural assumption is the following Dini condition:
\begin{equation} \label{eq:dini}
    \int_0^1 \frac{\omega(r)}{r} d r < \infty.
\end{equation}
Under this assumption, we will show that the functions $u^\pm$ are, in fact, differentiable up to the boundary on $\overline\opm$~(see Corollary~\ref{cor1}).
Condition \eqref{eq:dini} is more general than those previously considered, and it is the minimal assumption to have gradient estimates up to the interface, in some sense (see Remarks~\ref{rem:optimal} and \ref{rem:natural}).

The Dini condition has appeared extensively in the literature when studying the boundary behavior of solutions to elliptic boundary-value problems; see~\cite{DEK,DLK,HLW,HLZ,LXZ,MW,WLZ} and the references therein. 
For the classical Dirichlet and Neumann problems in a bounded domain $\Omega$,
boundary Lipschitz estimates have been obtained under an exterior $C^{1,\dini}$ condition on $\partial \Omega$ (see~\cite[Definition~1]{HLW}),
assuming the boundary data to be $C^{1,\dini}$ for the Dirichlet problem, and $C^{0,\dini}$ for the Neumann problem.
If, in addition, $\partial\Omega$ satisfies an interior $C^1$ condition (see~\cite[Definition~2]{HLW}), then it is well-known that solutions are differentiable on $\partial \Omega$. 
Furthermore, under a uniform two-sided $C^{1,\dini}$ condition on $\partial \Omega$ (see \cite[Definition 1.1]{DEK}), the regularity of solutions can be improved to $C^1$ up to the boundary.
For our problem~\eqref{eq:main}, to obtain Lipschitz estimates for $u$ on $\Gamma$, we need both $\Omega^+$ and $\Omega^-$ to satisfy an exterior $C^{1,\dini}$ condition. 
In particular, the interface $\Gamma=\partial\Omega^+$ 
can be written locally as the graph of a (uniformly) $C^{1,\dini}$ function,
and we expect $u^+$ and $u^-$ to be (continuously) differentiable on~$\Gamma$.

 \medskip

Next, we present our main results.
Here, for the definitions of the Dini spaces $C^{1, \dini}$ and $C^{0, \dini}$ and the Dini norm $\|\cdot\|_{C^{0,\dini}}$, we refer the reader to Section~\ref{sec:preliminaries}.

\begin{thm}[Optimal Lipschitz regularity] \label{thm1}
Let $\Gamma$ be a $C^{1, \dini}$ interface and let $g$ be a $C^{0, \dini}$ function on $\Gamma$.
Then the unique weak solution to~\eqref{eq:main} 
is Lipschitz continuous on $\overline{\Omega}$.
Moreover, we have the estimate
\[
\|u \|_{C^{0,1}(\overline{\Omega})} \leq C  \|g\|_{C^{0,\dini}(\Gamma)}, 
\]
where $C>0$ is a constant depending only on $n$, $\Omega$, and $\Gamma$. 
\end{thm}

By classical interior and boundary estimates for harmonic functions, it suffices to prove the Lipschitz estimate near $\Gamma$. 
Theorem~\ref{thm1} will be a consequence of the following pointwise gradient estimates on this interface.
Again, we give the definitions of the \emph{pointwise} Dini spaces $C^{1,\dini}(x_0)$, $C^{0,\dini}(x_0)$, and their associated norms in Section~\ref{sec:preliminaries}.
Below, $\omega_{f,x_0}$ always denotes the modulus of continuity of a function $f$ at the point $x_0$ 
(see Definitions~\ref{def:g} and~\ref{def:dinispaces}).

\begin{thm}[Pointwise gradient estimates on $\Gamma$] \label{thm2}
Let $\Gamma = \{x = (x', x_n)\in B_1 : x_n = \psi(x')\}$ for some function $\psi \colon \R^{n-1} \to \R$,
and let $g$ be a function on $\Gamma$.
Assume that 
$0\in \Gamma$, $\psi \in C^{1,\dini}(0)$, and $g\in C^{0,\dini}(0)$,
and let $u \in C(B_1)$ be a bounded distributional solution to 
\[
\Delta u = g \, \dhr \qquad \hbox{in}~B_1.
\]
Then there exist linear polynomials $l^\pm(x)= a^\pm \cdot x + b$,  with $a^\pm \in \Rn$ and $b\in \R$, 
a radius $0 < R \leq 1$ depending only on $\gradiniz$,
and a function $\sigma(r)$ depending only on $n$, 
$\|u\|_{L^{\infty}(B_1)}$, $|g(0)|$, $\gradiniz$, and $\gdiniz$, with $\sigma(r) \to 0$ as $r \to 0$, 
such that
\[
\sup_{\opm \cap B_r} |u^\pm - l^\pm| \leq r \sigma(r) \qquad \text{ for all } r \leq R.
\]
Furthermore, we have that
\[
a^+-a^-=g(0)e_n
\]
and the following estimate holds:
\[
R|a^{+}| + R|a^{-}| + |b| \leq C \left(\|u\|_{L^{\infty}(B_R)} + R|g(0)| + R \int_{0}^{R} \frac{\gdiniz(r)}{r} dr \right),
\]
for some constant $C > 0$ depending only on $n$.
\end{thm}

The following result follows directly from Theorem~\ref{thm2}, noting that $a^\pm = \nabla u^\pm(0)$, so the proof is omitted. 

\begin{cor}[Pointwise differentiability] \label{cor1}
Let $\Gamma$ be $C^{1, \dini}(x_0)$ (i.e., $\Gamma$ is the graph of a $C^{1,\dini}(x_0')$ function; see Section~\ref{sec:preliminaries}) and let $g \in C^{0, \dini}(x_0)$ for some $x_0 \in \Gamma$.
Then the weak solution $u$ to~\eqref{eq:main} is smooth in $\overline\Omega\setminus\Gamma$ and the functions $u^\pm: \overline{\opm} \to \R$ are differentiable at $x_0$.
In particular, if the above conditions are satisfied at every point $x_0 \in \Gamma$, then the functions $u^{\pm}$ are differentiable up to the boundary.
\end{cor}

Assuming a \emph{uniform} Dini condition on $\Gamma$ and $g$ (see Definitions~\ref{def:g}~\ref{def:gunif} and~\ref{def:dinispaces}~\ref{def:1dini}), we obtain that $u^\pm$ are \textit{continuously} differentiable up to the boundary.
The proof is based on a rather standard argument of patching the interior estimates for harmonic functions and the boundary estimates given in Theorem~\ref{thm2}, and follows from an application of Proposition~\ref{prop:patching} (see Appendix~\ref{sec:c1}).

\begin{cor}[Piecewise $C^1$ up to the boundary]\label{cor:piece}
Let $\Gamma$ be a $C^{1, \dini}$ interface and let $g$ be a $C^{0, \dini}$ function on $\Gamma$.
Let $u$ be the weak solution to~\eqref{eq:main}.
Then $u^\pm \in C^1(\overline{\Omega^{\pm}})$.
\end{cor}

\begin{rem}
\label{rem:dong}
We point out that Corollary~\ref{cor:piece} has been obtained by Dong~\cite{D} for operators in divergence form with Dini mean oscillation coefficients.
In particular, the global Lipschitz continuity in Theorem~\ref{thm1} can be deduced from the results in \cite{D}. 
Nonetheless, our approach is based on novel pointwise estimates on the interface (Theorem~\ref{thm2}) and, additionally, 
we provide sharp counterexamples highlighting the minimality of the assumptions (see Section~\ref{sec:counterexamples}).
Our Corollary~\ref{cor1} is interesting in its own right and does not follow from the approach in \cite{D}.
\end{rem}

Our strategy to prove Theorem~\ref{thm2} is inspired by an approximation technique developed by Caffarelli~\cite{C1,C2} in the context of fully nonlinear elliptic equations.
In our setting, the main idea is to view the equation $\Delta u= g \, \dhr$ as a perturbation of the ``flat'' problem $\Delta v = g(0) \mes{T_0 \Gamma}$ in a small ball $B_r$ centered at $0 \in \Gamma$, 
where $T_0 \Gamma$ denotes the tangent plane of $\Gamma$ at $0$.
The key step is to show that if $u=v$ on $\partial B_r$, then the error $u-v$ becomes sufficiently small in a smaller ball $B_{\rho r}$, with $\rho <1$. 
An iteration procedure then allows us to transfer the regularity of the well-behaved function $v$ to the solution $u$. 
The principal difficulty in estimating the size of $u-v$ lies in the fact that $\Delta u$ and $\Delta v$ are supported on different surfaces, $\Gamma$ and $T_0 \Gamma$, respectively.
Hence, it is not straightforward to compare these distributions. 
Nonetheless, since $\Gamma \sim T_0 \Gamma$ and $g\sim g(0)$ in $B_r$ as $r\to 0$, we expect $|u-v|\sim 0$.
This is the content of our key Lemma~\ref{keylemma}, 
where we quantify the closeness of $u$ and $v$ in terms of the moduli of continuity of $\Gamma$ and $g$.
We give a direct proof
using the representation formula for solutions to Poisson's equation in terms of the Green's function of the Laplacian (see Theorem~\ref{thm:existence}). 
We point out that this method is different than the one in~\cite{CSCS}, where the mean value property and the maximum principle are used.\medskip

Our results for the Laplacian are expected to be true for more general operators in divergence form.
For divergence operators with $C^{\alpha}$ coefficients, assuming $\Gamma \in C^{1,\alpha}$ and $g\in C^{\alpha}$, M\"uller~\cite{Muller3} has proved the Lipschitz continuity of the solution, while Dong~\cite{D} has further obtained the piecewise $C^{1,\alpha}$ regularity (see also Remark~\ref{rem:dong}).
Our key Lemma~\ref{keylemma} relies on estimates for both the Green's function and the fundamental solution of the Laplacian.
It would be natural to try to use fine estimates for the Green's function of general operators, as in the work of Kim and Sakellaris~\cite{KS} for divergence operators with Dini mean oscillation coefficients. 
This is an open problem.\medskip

The paper is organized as follows.
In Section~\ref{sec:preliminaries}, we introduce the necessary notation and definitions that we will use throughout this paper, as well as some preliminary results. 
Section~\ref{sec:counterexamples} provides counterexamples to Lipschitz regularity when $\Gamma$ is only a $C^1$ interface or $g$ is merely continuous. 
In Section~\ref{sec:key}, we prove the key error estimate using Green's function for the Laplacian. 
Finally, we give the proof of our main results in Section~\ref{sec:proofs}. 
The appendices contain a useful abstract result regarding the $C^1$ regularity up to the boundary (see Appendix~\ref{sec:c1}), and a counterexample illustrating the difference between pointwise and uniform Dini spaces (see Appendix~\ref{sec:pointwise}).

\section{Preliminaries} 
\label{sec:preliminaries}

Given a point $x\in \Rn$, we write $x=(x',x_n)$, where $x'\in \R^{n-1}$ and $x_n \in \R$. 
For $r>0$ and $x_0\in \Rn$, we denote the Euclidean ball of radius $r$ centered at $x_0$
by $B_r(x_0) = \{x \in \R^n \colon |x-x_0| < r\}$ in $\R^n$ and $B_r'(x_0')=\{x' \in \R^{n-1} \colon |x'-x_0'|<r\}$ in $\R^{n-1}$. If $x_0=0$, we simply write $B_r$ and $B_r'$.
 
\begin{defn}\label{def:dini}
    A continuous nondecreasing function $\omega : [0,\infty)\to [0,\infty)$ is called a \emph{Dini function}
    if there exists some $R_0>0$ such that
    $$
\int_0^{R_0} \frac{\omega(r)}{r} \, dr<\infty.
    $$
\end{defn}

Note that the H\"{o}lder modulus $\omega(r)=r^\alpha$ is a Dini function, for any  $\alpha \in (0,1)$.
The Dini class also includes logarithmic powers of the form $\omega(r)= |\log r|^{-1-\alpha}$, with $\alpha > 0$.\medskip

The following characterization will be useful in future proofs.

\begin{lem}\label{lem:summable}
A continuous and nondecreasing function $\omega : [0,\infty)\to [0,\infty)$
is a Dini function if and only if
there is a $R_0 > 0$ such that
\begin{equation*} 
\sum_{j=0}^\infty \omega(\rho^j R_0) < \infty \qquad \text{ for all }\rho\in (0,1).
\end{equation*}
\end{lem}

\begin{proof}
By monotonicity, on the one hand, we have that
\[
\int_0^{R_0} \frac{\omega(r)}{r} d r =  \sum_{j=0}^\infty \int_{\rho^{j+1}{R_0}}^{\rho^j{R_0}} \frac{\omega(r)}{r}\, dr \leq 
\sum_{j=0}^\infty \omega(\rho^{j}{R_0}) \int_{\rho^{j+1}{R_0}}^{\rho^j{R_0}} \frac{1}{r}\, dr
\leq \frac{1}{\rho} \sum_{j=0}^\infty  \omega(\rho^j{R_0}),
\]
and on the other hand
\[
\int_0^{R_0} \frac{\omega(r)}{r} d r \geq
\sum_{j=0}^\infty \omega(\rho^{j+1}{R_0}) \int_{\rho^{j+1}{R_0}}^{\rho^j{R_0}} \frac{1}{r}\, dr
\geq (1-\rho) \sum_{j = 1}^\infty  \omega(\rho^j{R_0}).
\]
It follows that the integral and the series are comparable for each $\rho \in (0,1)$, and hence the two conditions are equivalent.
\end{proof}

Next, we give the definitions of the pointwise and uniform Dini spaces.

\begin{defn}
\label{def:g}
Let $g$ be a continuous function defined on a bounded $C^1$ hypersurface $\Gamma \subset \R^n$.
For $x_0\in \Gamma$, we define the modulus of continuity of $g$ at $x_0$ as
$$
\omega_{g,x_0}(r) := \sup_{x\in \Gamma \cap B_r(x_0)} |g(x)-g(x_0)| \quad \hbox{for}~r > 0.
$$
\begin{enumerate}[label=(\roman*)]
\item We say that $g$ is \emph{Dini continuous at $x_0$}, and write $g\in C^{0,\dini}(x_0)$, if $\omega_{g,x_0}$ is a Dini function. 
Moreover, we define 
\[
\|g\|_{C^{0,\dini}(x_0)} :=|g(x_0)| + [g]_{C^{0,\dini}(x_0)},
\]
where the seminorm is given by
$$
[g]_{C^{0,\dini}(x_0)} := \int_0^{R_{0}} \frac{\omega_{g,x_0}(r)}{r}\, dr,
$$
with $R_0 > 0$ as in Definition~\ref{def:dini}.
\item 
\label{def:gunif}
We say that $g$ is \emph{Dini continuous on $\Gamma$}, and write $g\in C^{0,\dini}(\Gamma)$, if 
$$\omega_g(r) := \sup_{x_0\in \Gamma} \omega_{g,x_0}(r)$$ is a Dini function. Moreover, we define 
\[
\|g\|_{C^{0,\dini}(\Gamma)} :=\|g\|_{L^\infty(\Gamma)} + [g]_{C^{0,\dini}(\Gamma)},
\]
where the seminorm is defined as above, replacing $\omega_{g,x_0}$ with $\omega_g$.
\end{enumerate}
\end{defn}

\begin{rem} \label{rem:dini}
Note that $g \in C^{0,\dini}(x_0)$ for all $x_0 \in \Gamma$ is a weaker property than being in $C^{0,\dini}(\Gamma)$. 
We give an example illustrating this fact in Lemma~\ref{lem:dini} (see Appendix~\ref{sec:pointwise}).
\end{rem}

\begin{defn} \label{def:dinispaces}
Let $\Gamma\subset \R^n$ be a bounded $C^1$ hypersurface.
\begin{enumerate}[label=(\roman*)]
\item 
For $x_0 \in \Gamma$, we say that $\Gamma$ is $C^{1,\dini}(x_0)$
if there is a radius $r_{0} > 0$ and a $C^1$ function $\psi = \psi_{x_0} \colon \R^{n-1} \to \R$ such that, after a rotation of the coordinates, we have 
\[
\Gamma \cap B_{r_0}(x_0) = \big\{ x = (x',x_n) \in B_{r_0}(x_0) : x_n = \psi(x') \big\},
\]
with 
$\nabla \psi \in C^{0,\dini}(x_0')$,
that is,
$$
\omega_{\nabla \psi,x_0}(r) := \sup_{x\in \Gamma \cap B_r(x_0)} |\nabla \psi (x')-\nabla \psi(x'_0)| \quad \hbox{for}~r > 0
$$
is a Dini function.
We write $\psi \in C^{1,\dini}(x_0')$.
\item 
\label{def:1dini}
We say that $\Gamma$ is a $C^{1,\dini}$ interface if $\Gamma$ is $C^{1,\dini}(x_0)$ for all $x_0 \in \Gamma$, with $r_0 >0$ above independent of $x_0$, and with
\[
\omega_{\Gamma}(r) := \sup_{x_0 \in \Gamma} \omega_{\nabla \psi_{x_0}, x_0} (r)
\]
a Dini function.
\end{enumerate}
\end{defn}

\begin{rem}
By a rotation, we may assume that $\nabla \psi(x_0') = 0$, after
taking $r_0 > 0$ smaller if necessary.
We will always assume this condition.
\end{rem}

Finally, we discuss several important facts concerning Poisson's equation with a surface measure.
The equation in~\eqref{eq:main} can be understood in the sense of distributions as follows:
\begin{defn}[Weak solution]\label{def:solution}
We say that $u\in C(\Omega)$ is a \textit{distributional solution} to the equation 
\begin{equation} \label{eq:disteq}
\Delta u = g \mes{\Gamma} \quad\hbox{in}~\Omega
\end{equation}
 if for any $\varphi\in C^\infty_c(\Omega)$, we have
$$\int_\Omega u\Delta\varphi\,dx=\int_\Gamma g\varphi\,d H^{n-1}.$$
Furthermore, we say that $u \in C(\overline\Omega)$ is a \textit{weak solution} to the problem~\eqref{eq:main} if $u$ is a distributional solution to \eqref{eq:disteq} and $u=0$ on $\partial \Omega$. 
\end{defn}

Existence, uniqueness, and Log-Lipschitz regularity of weak solutions to~\eqref{eq:main} was established in~\cite{CSCS} using Green's functions. 
We recall that, for $x$ and $y$ in $\Omega$ with $x \neq y$, the Green's function for the Laplacian in $\Omega$ has the form
$$
G(x,y)= \Phi(x-y) - h^x(y).
$$
Here, $\Phi(x)$ is the fundamental solution of the Laplacian (with $\Delta \Phi=\delta_0)$ given by
$$
\Phi(x) = 
\begin{cases}
\frac{1}{2\pi}\log|x| & \hbox{if}~n=2,\\
\frac{1}{(2-n)\alpha_{n-1}} \frac{1}{|x|^{n-2}} & \hbox{if}~n\geq 3,
\end{cases}
$$
where we have written $\alpha_{n-1}:=H^{n-1}(\partial B_1)$,
and $h^x$ is the harmonic extension of $\Phi(x - \cdot)|_{\partial \Omega}$ to $\Omega$. For completeness, we state here the result. 

\begin{thm}[Theorem~2.2 in~\cite{CSCS}] \label{thm:existence}
Let $\Omega\subset \Rn$, with $n\geq 2$, be a smooth bounded domain. Let $\Gamma$ be a Lipschitz interface and $g\in L^\infty(\Gamma)$.
Then the unique weak solution $u\in C(\overline\Omega)$ to~\eqref{eq:main} is given by
\begin{equation*}
u(x)=\int_\Gamma G(x,y)g(y)\,dH^{n-1}(y) \qquad\hbox{for}~x\in \Omega,
\end{equation*}
where $G(x,y)$ is the Green's function for the Laplacian in $\Omega$. 
Furthermore, $u\in {\rm LogLip}(\overline\Omega)$ and
$$\|u\|_{L^\infty(\Omega)}\ + [u]_{{\rm LogLip}(\overline{\Omega})}\leq C\|g\|_{L^\infty(\Gamma)},$$
where $C>0$ depends only on $n$, $\Omega$, and $\Gamma$.
\end{thm}
Recall that a bounded function $u:\overline{\Omega}\to\R$ is said to be in the space ${\rm LogLip}(\overline\Omega)$ if
$$[u]_{{\rm LogLip}(\overline\Omega)}=\sup_{\substack{x,y\in\overline\Omega\\x\neq y}}\frac{|u(x)-u(y)|}{|x-y||\log|x-y||}<\infty.$$

\section{Counterexamples to Lipschitz regularity} \label{sec:counterexamples}

In this section, we prove that if $\Gamma$ is only a $C^1$ interface or $g$ is merely continuous, then weak solutions to~\eqref{eq:main} need not be Lipschitz. 
The counterexamples hold for $n\geq2$.

\begin{thm}[$\Gamma$ only $C^1$]\label{thm:counterexample}
Let $\Omega = B_1 \subset \R^n$, with $n \geq 2$.
Let $\Gamma \subset B_1$ be any closed $C^1$ hypersurface satisfying
\[
\Gamma \cap B_{1/4} = \left\{x =(x', x_n) \in \R^n \colon x_n = \frac{|x'|}{|\log|x'||}, \quad x' \in B_{1/4}' \subset \R^{n-1} \right\},
\]
and let $g \equiv 1$.
Then the unique weak solution $u$ to~\eqref{eq:main} is not locally Lipschitz in $B_1$.
\end{thm}
\begin{proof}
We will prove that the partial derivative $\partial_{x_n} u$ blows up as we approach $0$ along the vertical ray
$\{\vep e_n = (0, \vep)\}_{\vep > 0}$, 
thereby proving that $|\nabla u| \notin L^{\infty}_{\rm loc}(B_1)$ and hence $u \notin C^{0,1}_{\rm loc}(B_1)$.
More precisely, we claim that
\begin{equation}
\label{eq:lul}
\lim_{\vep \downarrow 0} \partial_{x_n} u(0, \vep)  = - \infty.
\end{equation}

By Theorem~\ref{thm:existence}, we have the representation formula
\[
u(x) = \int_\Gamma G(x,y) d H^{n-1} (y)
= \int_{\Gamma} \Phi(x-y) d H^{n-1} (y) - \int_{\Gamma} h^{x}(y) d H^{n-1} (y) 
\qquad\hbox{for}~x\in B_1.
\]
Assume that $ 0 < \vep \leq 1/16$.
Since $u$ is smooth outside $\Gamma$, we have
\begin{equation}
\label{eq:der}
- \partial_{x_n} u(0, \varepsilon) 
= \frac{1}{\alpha_{n-1}} \int_{\Gamma} \frac{y_n - \vep}{|(y', y_n-\vep)|^{n}} d H^{n-1} (y) 
+ \int_{\Gamma} \partial_{x_n} h^{x}(y)\big|_{x = (0, \vep)} d H^{n-1} (y).
\end{equation}
The function $x \mapsto h^{x}(y) = \Phi\big(|x|(y - \frac{x}{|x|^2})\big)$ is smooth in $B_1$ and hence the second integral in~\eqref{eq:der} is bounded.
Thus, we must show that the first integral in~\eqref{eq:der} diverges to infinity as $\vep \downarrow 0$.

We split $\Gamma$ into two disjoint parts: $\Gamma \setminus B_{1/4}$ and $\Gamma \cap B_{1/4}$.
First, we show that the integral outside $B_{1/4}$ is bounded.
To see this, note that if $|y| = |(y', y_n)| \geq 1/4$, then  $|y'| \geq 1/8$ or $|y_n| \geq 1/8$.
In the first case, we have that $|(y', y_n - \varepsilon)| \geq |y'| \geq 1/8$,
and in the second case $|(y', y_n - \varepsilon)| \geq |y_n| - \varepsilon \geq 1/16$.
It follows that 
\[
\begin{split}
\left|\int_{\Gamma \setminus B_{1/4}} \frac{y_n - \varepsilon}{|(y', y_n-\varepsilon)|^{n}} d H^{n-1} (y) \right|
\leq \int_{\Gamma \setminus B_{1/4}} \frac{|y_n| + \varepsilon}{|(y', y_n-\varepsilon)|^{n}} d H^{n-1} (y)
\leq 2 \cdot 16^n H^{n-1}(\Gamma \setminus B_{1/4}),
\end{split}
\]
which is bounded independently of $\vep$.

Next, we analyze the integral inside $B_{1/4}$.
Let $B_r'$ be the Euclidean ball in $\R^{n-1}$ with radius $r$.
Letting $\psi(r) := \frac{r}{|\log r |}$ for $r > 0$,
the surface integral on $\Gamma \cap B_{1/4}$ can be written in coordinates as
\begin{equation}
\label{eq:comp}
\begin{split}
&\int_{\Gamma \cap B_{1/4}} \frac{y_n - \vep}{|(y', y_n-\vep)|^{n}} d H^{n-1} (y) \\
& \quad = \int_{B_{1/4}'} \frac{\psi(|y'|) - \vep}{|(y', \psi(|y'|) -\vep)|^{n}} \sqrt{1 + \psi'(|y'|)^2} d y'\\
& \quad =  H^{n-2}(\partial B_{1}') \int_{0}^{1/4} \frac{(\psi(r) - \vep) r^{n-2}}{\left( r^2 + (\psi(r) -\vep)^2 \right)^{n/2}} \sqrt{1 + \psi'(r)^2} d r.
\end{split}
\end{equation}
Since $\psi$ is monotone increasing on $(0, 1/4]$ and $\psi(1/4)>1/16\geq \vep$, 
we can define
$r_{\vep}$ to be the unique $0 < r_{\vep} < 1/4$ such that $\psi(r_{\vep}) = \vep$.
Splitting the last integral in~\eqref{eq:comp} 
into the two intervals $(0, r_{\vep})$ and $(r_{\vep}, 1/4)$,
and using that $0< \psi'(r) \leq 2$, we can bound it from below by
\begin{equation}
\label{eq:two}
\begin{split}
&\int_{0}^{1/4} \frac{(\psi(r) - \vep) r^{n-2}}{\left( r^2 + (\psi(r) -\vep)^2 \right)^{n/2}} \sqrt{1 + \psi'(r)^2} d r\\
& \quad \geq 
- \sqrt{5} \int_{0}^{r_{\vep}} \frac{(\vep - \psi(r)) r^{n-2}}{\left( r^2 + (\psi(r) -\vep)^2 \right)^{n/2}} d r
+ \int_{r_{\vep}}^{1/4} \frac{(\psi(r) - \vep) r^{n-2}}{\left( r^2 + (\psi(r) -\vep)^2 \right)^{n/2}} d r.
\end{split}
\end{equation}

For the first integral, using that 
$\vep - \psi(r) = \psi(r_{\vep}) - \psi(r) = r\big( \frac{r_{\vep}}{r |\log r_{\vep}|} - \frac{1}{|\log r|}\big)$,
we have
\begin{equation}
\label{eq:what}
\begin{split}
&0 \leq \int_{0}^{r_{\vep}} \frac{(\vep - \psi(r)) r^{n-2}}{\left( r^2 + (\psi(r) -\vep)^2 \right)^{n/2}} d r\\
& \qquad \quad =
\int_{0}^{r_{\vep}} \frac{
\big(
\frac{r_{\vep}}{r |\log r_{\vep}|} - \frac{1}{|\log r|}
\big) 
}{\left( 1 + \big(\frac{r_{\vep}}{r |\log r_{\vep}|} - \frac{1}{|\log r|}\big)^2 \right)^{n/2} }  \frac{d r}{r} \\
& \qquad \quad =
\int_{0}^{1} \frac{\big(\frac{1}{t |\log r_{\vep}|} - \frac{1}{|\log (t \, r_{\vep})|}\big)}{\left( 1 + \big(\frac{1}{t |\log r_{\vep}|} - \frac{1}{|\log (t \, r_{\vep})|}\big)^2 \right)^{n/2} } \frac{d t}{t},\\
\end{split}
\end{equation}
where in the last line we rescaled by letting 
$t = r/r_{\vep}$. 
For $t \in (0, 1)$ fixed, since the function 
\[
r \mapsto \frac{1}{t |\log r|} - \frac{1}{ |\log (t r)|} 
= \frac{1}{t |\log (t r)|} \left(\frac{|\log t|}{|\log r|} + 1 - t \right) 
\]
is increasing in $r \in (0, 1)$,
and the function
\[
z \mapsto \frac{z}{(1 + z^2)^{n/2}}
\]
is increasing for $z \in (0, \frac{1}{\sqrt{n-1}})$,
it follows that the last integrand in~\eqref{eq:what} is monotone increasing in $\vep$, for $\vep$ small (note that $r_\vep\downarrow 0$ as $\vep \downarrow 0$). 
By monotone convergence, we deduce that
\[
\begin{split}
&0 \leq \lim_{\vep \downarrow 0} \int_{0}^{r_{\vep}} \frac{(\vep - \psi(r)) r^{n-2}}{\left( r^2 + (\psi(r) -\vep)^2 \right)^{n/2}} d r 
= \int_{0}^{1} \lim_{\vep \downarrow 0} \frac{\big(\frac{1}{t |\log r_{\vep}|} - \frac{1}{|\log (t \, r_{\vep})|}\big)}{\left( 1 + \big(\frac{1}{t |\log r_{\vep}|} - \frac{1}{|\log (t \, r_{\vep})|}\big)^2 \right)^{n/2} } \frac{d t}{t} = 0,
\end{split}
\]
and the first integral in~\eqref{eq:two} vanishes as $\vep\to 0$.

Finally, using that $0 \leq \frac{1}{|\log r|} - \frac{\vep}{r} \leq \frac{1}{\log 4}$ for $r_{\vep} < r < 1/4$, the second integral in \eqref{eq:two} can be bounded from below by
\[
\begin{split}
\int_{r_{\vep}}^{1/4} \frac{(\psi(r) - \vep) r^{n-2}}{\left( r^2 + (\psi(r) -\vep)^2 \right)^{n/2}} d r
&\geq
\int_{r_{\vep}}^{1/4} \frac{
\big(
\frac{1}{|\log r|} - \frac{\vep}{r}
\big) 
}{\left( 1 + \big(\frac{1}{|\log r|} - \frac{\vep}{r}\big)^2 \right)^{n/2} }  \frac{d r}{r}\\
&\geq 
\frac{1}{\left(1 + \frac{1}{\log^2 4}\right)^{n/2}}
\int_{r_{\vep}}^{1/4} \left(\frac{1}{|\log r|} - \frac{\vep}{r}\right)\frac{d r}{r},
\end{split}
\]
and by monotone convergence
\[
\lim_{\vep \downarrow 0} \int_{r_{\vep}}^{1/4} \left(\frac{1}{|\log r|} - \frac{\vep}{r}\right)\frac{d r}{r} = \int_{0}^{1/4} \frac{d r}{r |\log r|} = \infty.
\]
Since this is the only divergent integral, \eqref{eq:lul} follows.
\end{proof}

\begin{thm}[$g$ only continuous]\label{thm:counterexample2}
Let $\Omega = B_1 \subset \R^n$, with $n \geq 2$.
Let $\Gamma \subset B_1$ be a smooth closed hypersurface such that $\Gamma \cap B_{1/2} = \{x_n = 0\} \cap B_{1/2}$.
Consider the continuous function $\eta \colon (-1/2, 1/2) \subset \R \to \R$ given by
\[
\eta(t) = \begin{cases}
\frac{1}{|\log t|} & \text{ for } t\in (0, 1/2),\\
0 & \text{ for } t \in (-1/2, 0],
\end{cases}
\]
and let $g \colon \Gamma \to \R$ be any function on $\Gamma$, smooth outside $\Gamma \cap B_{1/2}$, and such that
\[
g(x) = \eta(x_1) \quad \text{ for } x \in \Gamma \cap B_{1/2}.
\]
Then the unique weak solution $u$ to~\eqref{eq:main} is not locally Lipschitz in $B_1$.
\end{thm}
\begin{proof}
We claim that the partial derivative $\partial_{x_1} u$ blows up as we approach $0$ along the vertical ray
$\{\varepsilon e_n = (0, \varepsilon)\}_{\varepsilon > 0}$.
As a consequence, we deduce that $|\nabla u| \notin L^{\infty}_{\rm loc}(B_1)$ and hence $u \notin C^{0,1}_{\rm loc}(B_1)$.

Fix $\vep$ sufficiently small such that $(0,\vep)\notin\Gamma$. 
Splitting the integral in the representation formula from Theorem~\ref{thm:existence},
the partial derivative $-\partial_{x_1} u$ at $(0, \varepsilon)$ can be written as
\[
-\partial_{x_1} u(0, \varepsilon) = \frac{1}{\alpha_{n-1}} \int_{B_{1/2}'} \frac{y_1 \, \eta(y_1)}{\left(|y'|^2 + \varepsilon^2 \right)^{n/2}} d y' + w(\varepsilon),
\]
where $w$ is a bounded smooth function.

We present the proof only for $n \geq 3$, while the calculation for $n = 2$ follows the same lines.
Let $B_r''$ be the Euclidean ball in $\R^{n-2}$ with radius $r$. 
Writing $y' = (y_1, y'') \in \R \times \R^{n-2}$, we have
\[
\begin{split}
\int_{B_{1/2}'} \frac{y_1 \, \eta(y_1)}{\left(|y'|^2 + \varepsilon^2 \right)^{n/2}} d y' 
&= \int_{0}^{1/2} \int_{B''_{\sqrt{1/4- y_1^2}}}   \frac{y_1 |\log y_1|^{-1}}{\left( y_1^2 + |y''|^2 + \varepsilon^2 \right)^{n/2}}  \, dy''  d y_1\\
& \geq \int_{0}^{1/4} \int_{B''_{1/4}} \, \frac{y_1 |\log y_1|^{-1}}{\left( y_1^2 + |y''|^2 + \varepsilon^2 \right)^{n/2}}  \, dy''  d y_1\\
& = H^{n-3}(\partial B_{1/4}'')\int_{0}^{1/4}\frac{y_1}{|\log y_1|} \left(\int_{0}^{1/4} \frac{ s^{n-3}}{\left( y_1^2 + s^2 + \varepsilon^2 \right)^{n/2}}\ ds  \right)  d y_1 ,
\end{split}
\]
where in the last line we used polar coordinates with $s = |y''|$.
Taking the limit as $\vep \downarrow 0$, by monotone convergence, it follows that
\[
\begin{split}
\lim_{ \varepsilon \downarrow 0}\,
&
\int_{0}^{1/4} \frac{y_1}{|\log y_1|}\left( \int_{0}^{1/4} \frac{ s^{n-3}}{\left( y_1^2 + s^2 + \varepsilon^2 \right)^{n/2}}\, ds  \right)d y_1\\
&\quad = 
\int_{0}^{1/4} \frac{y_1}{|\log y_1|}\left( \int_{0}^{1/4} \frac{s^{n-3}}{\left( y_1^2 + s^2 \right)^{n/2}}\, ds \right)  d y_1\\
& \quad = \int_{0}^{1/4} \frac{1}{y_1 |\log y_1|} \left( \int_{0}^{1/(4 y_1)} \frac{\, t^{n-3}}{\left( 1 + t^2 \right)^{n/2}} d t\right) d y_1,
\end{split}
\]
where in the last line we rescaled by letting $t = s/y_1$.
For $0< y_1 < 1/4$, we have
\[
\int_{0}^{1/(4 y_1)} \,\frac{\, t^{n-3}}{\left( 1 + t^2 \right)^{n/2}} d t \geq \int_{0}^{1} \,\frac{\, t^{n-3}}{\left( 1 + t^2 \right)^{n/2}} d t 
= \frac{1}{(n-2)2^{(n-2)/2}} > 0,
\]
and hence
\[
\begin{split}
\lim_{ \varepsilon \downarrow 0}\,
\int_{B_{1/2}'} \frac{y_1 \, \eta(y_1)}{\left(|y'|^2 + \varepsilon^2 \right)^{n/2}} d y' 
\geq 
\frac{H^{n-3}(\partial B_{1/4}'')}{(n-2)2^{(n-2)/2}}
\int_{0}^{1/4} \frac{d y_1}{y_1 |\log y_1|} = \infty.
\end{split}
\]
From the above computations, we deduce that $\displaystyle\lim_{\varepsilon \downarrow 0} \partial_{x_1} u (0, \varepsilon)= - \infty$ and the claim follows.
\end{proof}

\begin{rem}\label{rem:counterexample}
The fact that the Lipschitz regularity of the solution should fail for bounded or continuous $g$ can also be inferred from the classical elliptic theory.
To illustrate this, 
we restrict to a localized version of the problem where the interface is not closed.
In the case that $\Gamma$ is flat, e.g., $\Gamma = \{x_n = 0\}$, the transmission problem \eqref{eq:transmission}
decouples into the two Neumann problems
    \[
    \begin{cases}
    \Delta u^\pm = 0 & \text{ in } B_1^\pm,\\
    \partial_{x_n} u^\pm = \pm \frac{1}{2} g & \text{ on } \Gamma\cap B_1,\\
    u^\pm = 0 & \text{ on } \partial B_1^\pm \setminus \Gamma,
    \end{cases}
    \]
    where $B_1^\pm=B_1\cap\{\pm x_n>0\}$. 
    The $L^p$ and Schauder estimates for such problems yield, respectively,
\[
    g \in L^p \quad \Longrightarrow \quad  \nabla u \in L^p \qquad (1 < p < \infty),
    \]
    \[
    g \in C^{0,\alpha} \quad \Longrightarrow \quad  \nabla u \in C^{0,\alpha} \qquad (0 <\alpha < 1).
    \]
It is well-known that these implications do not hold in the extreme cases $p = \infty$ and $\alpha = 0$.
\end{rem}

\begin{rem}\label{rem:optimal}
Our counterexamples in Theorems~\ref{thm:counterexample} and~\ref{thm:counterexample2} are critical with respect to the Dini condition.
In both cases, the moduli of continuity at $0$ satisfy $\omega_{\psi}(r) \sim \omega_g(r) \sim \frac{1}{|\log r|}$,
which is not a Dini function, as the change of variables $z = |\log r|$ shows that
\[
\int_{0}^{\delta} \frac{d r}{r |\log r|} = \int_{|\log \delta|}^{\infty} \frac{d z}{z} = \infty,
\]
for all $\delta > 0$.
Combined with our main results, these examples imply the minimality of the Dini continuity assumption to have Lipschitz regularity of solutions to Poisson's equation.
\end{rem}

\section{A quantitative error estimate}
\label{sec:key}

As explained in the introduction, to prove the Lipschitz regularity of solutions to~\eqref{eq:main}, it suffices to establish local estimates close to the interface $\Gamma$.
Restricting to a small neighborhood of each point on $\Gamma$,
by translation, rotation, and rescaling, the surface $\Gamma$ is locally the graph of a function $\psi \colon \overline{B_1'} \subset \R^{n-1} \to \R$. 
The following assumptions correspond to such a normalized situation.\medskip

Fix $0<r\leq 1$.
Let $\psi\in C^{1}(\overline{B_r'}) \cap C^{1,\dini}(0)$ 
and $g\in C^0(\Gamma\cap \overline{B_r})\cap C^{0,\dini}(0)$, where
\[
\Gamma=\{x\in B_r : x_n = \psi(x')\}.
\]
Assume that $0 \in \Gamma$ and that the tangent plane of $\Gamma$ at $0$ is $T_0\Gamma = \{x_n = 0\}$, i.e.,
\[
\psi(0) = 0 \quad \text{ and } \quad \nabla \psi(0) = 0.
\]
By assumption, there are Dini functions $\gradiniz$ and $\gdiniz$ such that
\[
\hspace{-0.5cm}
|\nabla \psi(x')| \leq \gradiniz(|x'|) \qquad \text{ for all } x' \in B_{r}',
\]
and
\[
|g(x) - g(0)|  \leq \gdiniz(|x|) \qquad \hbox{ for all }~  x\in \Gamma \cap B_r.
\]
We also assume that
\[
\gradiniz(r)\leq 1
\]
by taking $r$ smaller if necessary.

Geometrically, the $C^{1,\dini}(0)$ regularity of $\psi$ means that the curved interface $\Gamma$ is (quantifiably) close to its tangent plane $T_0 \Gamma$.
Similarly, the $C^{0,\dini}(0)$ assumption means that the density $g$ is close to the constant value $g(0)$ near the origin.
Thus, we expect the solutions of the Poisson equation with curved $\Gamma$ and variable $g$ to be close
to the solutions of the analogous problem with a flat interface and a constant density.

Indeed, under the above assumptions, the following lemma shows that if $u$ and $v$ are distributional solutions of, respectively, $\Delta u= g \mes{\Gamma}$ and $\Delta v = g(0) \mes{\{x_n=0\}}$ in $B_r$, with $u=v$ on $\partial B_r$, then the error $w=u-v$ must be locally bounded by $r \omega(r)$, where
\begin{equation}
\label{def:modulus}
\omega(r) := \max\big\{|g(0)|\gradiniz(r), \, \gdiniz(r)\big\}.
\end{equation}

\begin{lem} \label{keylemma}
Let $w$ be the unique weak solution to the Dirichlet problem
\[
\begin{cases}
\Delta w = g \mes{\Gamma} - g(0) \mes{\{x_n = 0\}} & \text{in}~B_{r},\\
w = 0 & \text{on}~\partial B_{r}.
\end{cases}
\]
Assuming that $n \geq 3$,
for any $0<\rho \leq 1/2$, we have that 
\[
\sup_{B_{\rho r}} |w| \leq C r\omega(r),
\]
where $C>0$ depends only on $n$.
\end{lem}
\begin{proof}
By Theorem~\ref{thm:existence}, 
the solution $w$ can be written as
\begin{equation} \label{eq:repw}
\begin{split}
w(x) &= \int_{\Gamma \cap B_{r}} G(x, y) g(y) d H^{n-1} (y) - \int_{\{y_n = 0\} \cap B_{r}} \!\!\!\!\!\!\!\!G(x, y) g(0) d H^{n-1} (y),
\\
\end{split}
\end{equation}
where 
$
G(x, y) = \Phi(x - y) - h^x(y)
$
is the Green's function of the Laplace operator in the ball $B_r \subset \R^n$.
Here, we have
\[
h^{x}(y) 
= \Phi \left(r \frac{x}{|x|} - |x|\frac{y}{r}\right).
\]

Using~\eqref{eq:repw}, for each $x \in B_{\rho r}$, we can bound $w(x)$ by
\begin{equation}
\label{eq:suf}
\begin{split}
|w(x)| &\leq \left|\int_{\Gamma \cap B_{r}} G(x, y) (g(y) - g(0) ) d H^{n-1} (y)\right| \\
& \quad \quad + |g(0)|\left| \int_{\Gamma \cap B_{r}} G(x, y) d H^{n-1} (y) - \int_{\{y_n = 0\} \cap B_{r}} \!\!\!\!\!\!\!\!G(x, (y',0)) \sqrt{1 +|\nabla \psi(y')|^2} dy' \right|\\
& \quad \quad + |g(0)|\left| \int_{B'_{r}} G(x, (y',0))\left(\sqrt{1 +|\nabla \psi(y')|^2} - 1\right) d y' \right|
\end{split}
\end{equation}
and it suffices to estimate each of the terms on the right-hand side.

\vspace{3mm}
\noindent
\textbf{Step 1:} \textit{First and third terms in~\eqref{eq:suf}}.
In the graph coordinates $y' \mapsto (y', \psi(y'))$, the  measure of $\Gamma$ takes the form $\mes{\Gamma} = \sqrt{1 + |\nabla \psi(y')|^2} d y'$ for all $y' \in B_r'$.
Since $\nabla \psi$ is Dini continuous, we have
\begin{equation}
\label{eq:psidini}
    |\nabla \psi(y')| = |\nabla \psi(y') - \nabla \psi(0)| \leq \gradiniz(|y'|) \leq \gradiniz(r)
\end{equation}
for all $y' \in B_r'$.
Using that $\sqrt{1 + |\nabla \psi(y')|^2} \leq \sqrt{2}$ since $\gradiniz(r)\leq 1$, it follows that
\[
\begin{split}
\int_{\Gamma \cap B_{r}} |\Phi(x-y)| d H^{n-1} (y) &=  
c_n\int_{B_r'} \left(|x'- y'|^2 + |x_n - \psi(y')|^2 \right)^{\frac{2-n}{2}} \sqrt{1 + |\nabla \psi(y')|^2}d y' \\
&\leq C \int_{B_r'} |x' - y'|^{2-n} d y' \leq C \int_{B_{2r}'} |y'|^{2-n} d y' \leq C r.
\end{split}
\]
For $x \in B_{\rho r}$ and $y \in B_r$, we have 
\begin{equation}\label{low}
\left|r \frac{x}{|x|} - |x|\frac{y}{r}\right| \geq r - |x| \frac{|y|}{r} \geq r (1 - \rho),
\end{equation}
hence
\[
\|h^{x}\|_{L^{\infty}(B_{ r})} = \textstyle\|\Phi(r\frac{x}{|x|} - |x| \frac{\cdot}{r})\|_{L^{\infty}(B_{ r})} \leq C r^{2-n}(1-\rho)^{2-n}
\]
and therefore
\[
\begin{split}
\int_{\Gamma \cap B_{r}} |h^{x}(y)| d H^{n-1} (y) &
\leq \|h^{x}\|_{L^{\infty}(B_r)} H^{n-1}(\Gamma \cap B_r) \leq C(1-\rho)^{2-n} r \leq Cr.
\end{split}
\]
By Dini continuity of $g$, the first term in~\eqref{eq:suf} can be bounded by
\[
\begin{split}
\left|\int_{\Gamma \cap B_{r}} G(x, y) (g(y) - g(0) ) d H^{n-1} (y)\right|  &\leq 
\int_{\Gamma \cap B_{r}} \left|G(x, y)\right| d H^{n-1} (y) \, \gdiniz(r) \leq C r\gdiniz(r).
\end{split}
\]

To estimate the third term in~\eqref{eq:suf}, by 
\eqref{eq:psidini}, we have
\[
0 \leq \sqrt{1 +|\nabla \psi(y')|^2} - 1 \leq \frac{1}{2} |\nabla \psi(y')|^2 \leq 
\frac{1}{2} \gradiniz(r),
\]
and arguing as above we obtain
\[
\left| \int_{B'_{r}} G(x, (y',0))\left(\sqrt{1 +|\nabla \psi(y')|^2} - 1\right) d y' \right|
\leq C  r\gradiniz(r).
\]

\vspace{3mm}
\noindent \textbf{Step 2:} \textit{Second term in~\eqref{eq:suf}}.
We split the integral into a singular and regular part as follows
\begin{equation}
\label{eq:sef}
\begin{split}
&\left| \int_{\Gamma \cap B_{r}} G(x, y) d H^{n-1} (y) - \int_{\{y_n = 0\} \cap B_{r}} \!\!\!\!\!\!\!\!G(x, y) \sqrt{1 +|\nabla \psi(y')|^2} dy' \right|\\
& \quad \leq \left| \int_{\Gamma \cap B_{r}} \Phi(x- y) d H^{n-1} (y) - \int_{\{y_n = 0\} \cap B_{r}} \!\!\!\!\!\!\!\!\Phi(x- y) \sqrt{1 +|\nabla \psi(y')|^2} dy' \right|\\
& \quad \quad \quad+ \left| \int_{\Gamma \cap B_{r}} h^{x}(y) d H^{n-1} (y) - \int_{\{y_n = 0\} \cap B_{r}} \!\!\!\!\!\!\!\!h^{x}(y) \sqrt{1 +|\nabla \psi(y')|^2} dy' \right|.
\end{split}
\end{equation}

We bound the singular part in~\eqref{eq:sef} first.
Writing $\mes{\Gamma} = \sqrt{1 + |\nabla \psi(y')|^2}d y'$ and using that $\sqrt{1 + |\nabla \psi(y')|^2} \leq \sqrt{2}$, we have
\[
\begin{split}
&\left| \int_{\Gamma \cap B_{r}} \Phi(x- y) d H^{n-1} (y) - \int_{\{y_n = 0\} \cap B_{r}} \!\!\!\!\!\!\!\!\Phi(x- y) \sqrt{1 +|\nabla \psi(y')|^2} dy' \right|\\
& \quad \quad = \left| \int_{B'_{r}} \Big(\Phi(x'- y', x_n - \psi(y')) -\Phi(x- y)\Big) \sqrt{1 +|\nabla \psi(y')|^2} dy' \right|\\
& \quad \quad \quad \quad  \leq C \int_{B_r'} \Big| \Phi(x' - y', x_n - \psi(y')) - \Phi(x' - y', x_n)\Big| d y'.
\end{split}
\]
By the fundamental theorem of Calculus,
\[
\begin{split}
&\int_{B_r'} \Big| \Phi(x' - y', x_n - \psi(y')) - \Phi(x' - y', x_n)\Big| d y'\\
& \quad \quad = \int_{B_{r}'}  \left|\int_{0}^{\psi(y')} \partial_{x_n} \Phi(x' - y', x_n - t) d t\right| d y'\\
& \quad \quad \quad \quad \leq \int_{- r \gradiniz(r)}^{ r \gradiniz(r)}  \int_{B_{r}'}  |\partial_{x_n} \Phi(x'- y', x_n - t)| d y' d t,
\end{split}
\]
where in the last line we used that $|\psi(y')| \leq r \gradiniz(r)$ for $y' \in B'_{r}$ together with Fubini's theorem.
Since $B_r' \subset B'_{(1 + 2 \rho)r }(x') \subset B'_{2 r }(x')$ for $x \in B_{\rho r}$, the integral in $B_r'$ can be bounded by
\[
\begin{split}
\int_{B_{r}'} |\partial_{x_n} \Phi(x'- y', x_n - t)| d y' &\leq C \int_{B_{2r}'(x')} \frac{|x_n-t|}{\left( |x'-y'|^2 + |x_n - t|^2 \right)^{n/2}} d y'\\
&= C \int_{B_{\frac{2r}{|x_n-t|}}'}  \frac{1}{\left( |z'|^2 + 1 \right)^{n/2}} d z' \leq C \int_{0}^{\infty} \frac{\tau^{n-2}}{(1 + \tau^2)^{n/2}} d \tau \leq C.
\end{split}
\]
Therefore, it follows that
\begin{equation}
\label{eq:phiterm}
\begin{split}
&\left| \int_{\Gamma \cap B_{r}} \Phi(x- y) d H^{n-1} (y) - \int_{\{y_n = 0\} \cap B_{r}} \!\!\!\!\!\!\!\!\Phi(x- y) \sqrt{1 +|\nabla \psi(y')|^2} dy' \right| \leq 
C r \gradiniz(r).
\end{split}
\end{equation}

Next, we estimate the regular part in~\eqref{eq:sef}.
Using the expression for $h^{x}$ and by~\eqref{low} above, it follows that
\[
\|\partial_{y_n} h^{x}\|_{L^{\infty}(B_r)} 
\textstyle \leq \frac{|x|}{r} \big\|\nabla \Phi\big(r \frac{x}{|x|} - |x|\frac{\cdot}{r}\big)\big\|_{L^{\infty}(B_r)}
\leq \rho (1-\rho)^{1-n} r^{1-n} \leq C r^{1-n}
\]
for $x \in B_{\rho r}$.
By the mean value theorem and \eqref{eq:psidini}, we get
\[
\begin{split}
& 
\int_{B_r'} \left| h^{x}(y', \psi(y')) - h^{x}(y', 0) \right| d y'\\
& \quad \quad \leq C \|\partial_{y_n} h^x\|_{L^{\infty}(B_r)} \int_{B_r'} |\psi(y')| d y' 
\leq C  r \gradiniz(r),
\end{split}
\]
and hence, 
\begin{equation}
\label{eq:hterm}
\begin{split}
&\left| \int_{\Gamma \cap B_{r}} h^{x}(y) d H^{n-1} (y) - \int_{\{y_n = 0\} \cap B_{r}} \!\!\!\!\!\!\!\! h^{x}(y) \sqrt{1 +|\nabla \psi(y')|^2} dy' \right| \leq C   r \gradiniz(r).
\end{split}
\end{equation}

Combining~\eqref{eq:phiterm} and~\eqref{eq:hterm} in~\eqref{eq:sef}, we finally obtain
\[
\begin{split}
&\left| \int_{\Gamma \cap B_{r}} G(x, y) d H^{n-1} (y) - \int_{\{y_n = 0\} \cap B_{r}} \!\!\!\!\!\!\!\!G(x, y) \sqrt{1 +|\nabla \psi(y')|^2} dy' \right| \leq C r \gradiniz(r).
\end{split}
\]

\vspace{3mm}
 Steps $1$ and $2$ yield the estimate
\[
|w(x)| \leq C r \left( \gdiniz(r) + |g(0)| \gradiniz(r) \right)
\quad \text{ for all } x \in B_{\rho r},
\]
where $C$ depends only on $n$, thus proving the claim.
\end{proof}

\begin{rem}
For $n = 2$, the above proof leads to the weaker estimate
\[
\sup_{B_{\rho r}}|w| \leq 
C r |\log r| \omega(r).
\]
The additional logarithmic error arises only in Step 1.
Nonetheless, Theorems~\ref{thm1} and~\ref{thm2} can still be obtained from this estimate under the stronger assumption that $\omega(r)$ is a  Log-Dini function, that is, satisfying
\begin{equation}\label{eq:logdini}
\int_0^{r_0} \frac{|\log r| \, \omega(r)}{r}\, dr <\infty
\end{equation}
for some $r_0>0$.
To deduce our results in this case, assuming the moduli $\gradiniz$ and $\gdiniz$ to be Log-Dini functions, one must simply multiply the maximum $\omega(r) = \max\{|g(0)| \gradiniz(r), \gdiniz(r)\}$ by $|\log r|$ in the proofs from Section~\ref{sec:proofs}.

We point out that, since the H\"{o}lder modulus $\omega(r)=r^\alpha$ satisfies~\eqref{eq:logdini}, when $\Gamma$ is a $C^{1,\alpha}$ interface and $g \in C^{0,\alpha}$, our main result for $n \geq 2$ recovers results from~\cite{CSCS,Muller1,Muller2}.

The Log-Dini assumption for $n = 2$ is not the sharp condition to have Lipschitz continuity. 
This can be inferred from our counterexample in Theorem~\ref{thm:counterexample2} by modifying the function $\eta$ appearing there.
For instance, choosing $\eta(t) = \frac{1}{|\log t|^2}$ instead of $\eta(t) = \frac{1}{|\log t|}$ for $t \in (0,\frac{1}{2})$ yields a Dini function $g$, which is not Log-Dini, but gives rise to a Lipschitz solution.
This example also suggests that the Dini condition should be the optimal assumption; see~\cite{D}.
\end{rem}

\section{Proofs of main results}
\label{sec:proofs}

First, we prove the following discrete version of Theorem~\ref{thm2}.

\begin{thm}
\label{thm3}
Under the assumptions of Theorem~\ref{thm2}, 
there exist $0<\rho\leq 1/2$, depending only on $n$,
and $0 < R \leq 1$, depending only on $\gradiniz$,
such that for all integers $k \geq 0$, 
 there are linear polynomials $l_k^\pm(x) = a_k^\pm \cdot x + b_k$, with $a_k^\pm\in \Rn$ and $b_k\in \R$, satisfying 
 $$a_k^+-a_k^-=g(0)e_n,$$ 
such that
\begin{equation} \label{eq:mainest}
\sup_{\opm \cap B_{\rho^k R}} |u^\pm - l_k^\pm| \leq C_0 d_k \rho^k R
\end{equation}
and
\begin{equation}
\rho^{k-1}R |a_{k}^\pm - a_{k-1}^\pm| + |b_{k}-b_{k-1} | \leq C_1C_0 d_{k-1}\rho^{k-1} R,\label{eq:cauchy}
\end{equation}
where $a_{-1}^\pm=0$, $b_{-1}=0$,
$d_{-1}=d_0= \|u\|_{L^{\infty}(B_R)} R^{-1} + |g(0)|$,
and
\begin{equation} \label{eq:dk}
d_{k} =  
\max\{
\gdiniz(\rho^k R), \, |g(0)| \gradiniz(\rho^{k} R), 
\, \rho^{1/2} d_{k-1}\}
\quad \hbox{for}~k\geq 1,
\end{equation}
for some constants $C_0 > 0$ and $C_1 > 0$ depending only on $n$.
\end{thm}

\begin{proof}[Proof of Theorem~\ref{thm3}]
Recall that $\Gamma = \{x\in B_1 : x_n = \psi(x')\}$,
with $0\in \Gamma$, $\psi \in C^{1,\dini}(0)$, and $g\in C^{0,\dini}(0)$ (see Definitions~\ref{def:g} and~\ref{def:dinispaces}).
By rotation, we may assume that $\nabla \psi(0)=0$.
Let $0 < R \leq 1$ be such that 
\[
\gradiniz(R) \leq 1 \quad \text{ and } \int_{0}^{R} \frac{\gradiniz(t)}{t} d t \leq 1.
\]

Let $u \in C(B_1)$ be a bounded distributional solution to 
$$\Delta u = g\, \dhr \qquad \hbox{in}~B_1.$$
We write $\omega(r) = \max\{|g(0)| \gradiniz(r), \gdiniz(r)\}$ as in~\eqref{def:modulus}.

We will prove the theorem by induction. 
For $k=0$, let $a_0^\pm = \pm \tfrac{g(0)}{2}e_n$ and $b_0=0$. 
Then
\[
\sup_{\opm \cap B_R} |u^{\pm} - l_0^{\pm}| \leq \|u\|_{L^{\infty}(B_R)} + \frac{1}{2}|g(0)| R \leq C_0 d_0 R,
\]
and
\[
|a_0^\pm | \leq \frac{1}{2} |g(0)| \leq C_1C_0 d_{-1}\rho^{-1},
\]
for any $C_0,C_1 \geq 1$ and $0<\rho\leq 1/2$. 
Assume that~\eqref{eq:mainest}, \eqref{eq:cauchy}, and~\eqref{eq:dk} hold for some $k$. We need to prove that they also hold for $k+1$. 
Let $r= \rho^k R$ and 
denote $\Omega_r^\pm = \Omega^\pm \cap B_r$, $\Gamma_r = \Gamma \cap B_r$, and $B_r^{\pm} = B_r \cap \{\pm x_n > 0\}$.
We define $p_k$ as the piecewise linear polynomial given by 
$$
p_k = l_k^+\chi_{\overline{B_r^+}} + l_k^-\chi_{B_r^-},
$$
where $l_k^\pm$ are the linear approximations obtained by the induction hypothesis and satisfying the conditions in Theorem~\ref{thm3}. 
Notice that 
$(\nabla l_k^+ - \nabla l_k^-)\cdot e_n = (a_k^+-a_k^-)\cdot e_n=g(0)$  and $p_k\in C(\overline{B_r})$  satisfies
\begin{equation*} 
    \Delta p_k = g(0) \mes{\{x_n=0\}} \quad \hbox{in}~B_r,
\end{equation*}
in the distributional sense. Indeed, for any $\varphi\in C_c^\infty(B_r)$, by Green's identity, we have
\begin{align*}
\Delta p_k(\varphi) 
&= \int_{B_r^+}l_k^+  \Delta \varphi \, dx +
 \int_{B_r^-} l_k^- \Delta \varphi \, dx\\
 &=  - \int_{\{x_n=0\}\cap B_r} l_k^+ \partial_{x_n} \varphi \, d H^{n-1}+  \int_{\{x_n=0\}\cap B_r} \partial_{x_n} l_k^+ \varphi \, d H^{n-1} \\
&\qquad + \int_{\{x_n=0\}\cap B_r} l_k^- \partial_{x_n} \varphi \, d H^{n-1}-  \int_{\{x_n=0\}\cap B_r} \partial_{x_n} l_k^- \varphi \, d H^{n-1} \\
& = \int_{\{x_n=0\}\cap B_r} (l_k^--l_k^+) \partial_{x_n}\varphi \, d H^{n-1}+ \int_{\{x_n=0\}\cap B_r}( \partial_{x_n}l_k^+ - \partial_{x_n}l_k^-) \varphi \, d H^{n-1}\\
 &=  \int_{\{x_n=0\}\cap B_r} g(0) \varphi \, d H^{n-1},
\end{align*}
where in the last line we have used that for $x \in \R^n$ such that $x_n = 0$, we have
\[
l_k^{+}(x) - l_{k}^{-}(x) = (a_k^{+} - a_k^{-})\cdot x = g(0) e_n \cdot x = g(0) x_n = 0.
\]
Furthermore, 
since $\psi(|x'|) \leq r \gradiniz(r)$ for $|x'| \leq r$,
we have 
\begin{equation}\label{eq:polyclose}
\sup_{\Omega^\pm_{r}}|p_k - l_k^\pm| = \sup_{\Omega^\pm_{ r}\cap B_{ r}^\mp}|(a_k^- - a_k^+) \cdot x| 
= \sup_{\Omega^\pm_{ r}\cap B_{r}^\mp}|g(0)| |x_n| \leq r |g(0)| \gradiniz( r) \leq r \omega(r),
\end{equation}
where we used the compatibility condition $a_k^+ - a_k^- = g(0) e_n$ and the fact that
\[
\Omega_{ r}^\pm \cap B_{r}^\mp \subset \big\{|x_n| \leq |\psi(x')| \big\} \cap B_r.
\]
Therefore, by~\eqref{eq:mainest}, \eqref{eq:polyclose}, and \eqref{eq:dk}, 
since $C_0 \geq 1$,
it follows that
\begin{equation}\label{eq:polyclose2}
\begin{split} 
   \sup_{\Omega_r^\pm \cap B_r^\mp} |u^\pm - p_k|
 &\leq  \sup_{\Omega_r^\pm \cap B_r^\mp}|u^\pm-l_k^\pm| +  \sup_{\Omega_r^\pm \cap B_r^\mp}|p_k - l_k^\pm |\\
& \leq C_0 r d_k + r\omega(r)
 \leq  2C_0 r d_k.
 \end{split}
\end{equation} 
Hence, by~\eqref{eq:mainest} and~\eqref{eq:polyclose2}, we get
\begin{equation}\label{eq:polyclose3}
    \begin{split} 
\sup_{{B_r}} |u - p_k| & \leq \max\Big\{ \sup_{\Omega_r^+ \cap B_r^+} |u^+-l_k^+|, \sup_{\Omega_r^+ \cap B_r^-} |u^+-p_k|, \sup_{\Omega_r^- \cap B_r^+} |u^--p_k|, \sup_{\Omega_r^- \cap B_r^-} |u^--l_k^-|\Big\} \\
& \leq 2 C_0 r d_k.
\end{split}
\end{equation}

Let $v\in C^\infty(B_r)\cap C(\overline{B_r})$ be the classical solution to the Dirichlet problem
$$
\begin{cases}
\Delta v =0 & \hbox{in}~B_r,\\
v=u-p_k & \hbox{on}~\partial B_r.
\end{cases}
$$
Define the linear polynomial $l(x)=a\cdot x + b$, with $a=\nabla v(0)$ and $b= v(0)$.
Then by interior estimates for harmonic functions and~\eqref{eq:polyclose3}, there is some dimensional constant $C_2 = C_2(n)>0$ such that
\begin{equation}
\begin{split}\label{eq:estv}
\sup_{B_{\rho r}} | v - l | & \leq \| D^2 v\|_{L^\infty(B_{\rho r})} (\rho r)^2 \leq \frac{C_2}{(1-\rho)^2 r^2} \sup_{\partial B_r} |u-p_k| (\rho r)^2 \\
& \leq \frac{4C_2}{ r^2} ( 2 C_0 r d_k )(\rho r)^2
\leq \left(8 C_2 \rho^{1/2} \right) C_0 d_{k+1} \rho^{k+1} R\\
&\leq \frac{1}{3}C_0 d_{k+1} \rho^{k+1} R,
\end{split}
\end{equation}
by choosing $\rho = \rho(n)>0$ sufficiently small so that $8 C_2\rho^{1/2} \leq 1/3$. 
Moreover, by~\eqref{eq:polyclose3} and standard interior estimates for harmonic functions, we have
\begin{equation}\label{eq:estcoef}
r |a| + |b| \leq r \|\nabla v\|_{L^\infty{(B_{r/2}})} +\| v\|_{L^\infty{(B_{r/2}})}  \leq C_2 \sup_{\partial B_r} |u-p_k| \leq C_1 C_0 d_k r,
\end{equation}
where $C_1= 2C_2$.\medskip

Let $w\in C(\overline{B_r})$ be the function given by
$$w=u-p_k-v.$$
Using the equations satisfied by $u$ and $p_k$, we get that $w$ is a weak solution to
\[
\begin{cases}
\Delta w = g \mes{\Gamma} - g(0) \mes{\{x_n = 0\}} & \text{in}~B_{r},\\
w = 0 & \text{on}~\partial B_{r}.
\end{cases}
\]
By Lemma~\ref{keylemma} and \eqref{eq:dk}, 
there is a dimensional constant $C_3 = C_3(n) > 0$ such that
\begin{equation}\label{eq:estw}
\begin{split}
\sup_{B_{\rho r}} |w| &
\leq C_3 r \omega(r)  = C_3 r \max\left\{|g(0)| \gradiniz(r), \gdiniz(r)\right\}
\leq \frac{C_3}{\rho^{1/2}} R \rho^k \rho^{1/2}d_k \\ 
&\leq   \Big(\frac{C_3}{\rho^{3/2}C_0}\Big) C_0 d_{k+1} \rho^{k+1} R \\
& \leq \frac{1}{3} C_0 d_{k+1} \rho^{k+1} R,
\end{split}
\end{equation}
where in the last line we chose $C_0 = C_0(n)>0$ sufficiently large such that $3C_3\leq \rho^{3/2}C_0$.

Define the linear approximations $l_{k+1}^\pm(x)=a_{k+1}^\pm\cdot x + b_{k+1}$ as
\begin{equation}\label{eq:lk1}
l_{k+1}^\pm = l_k^\pm + l.
\end{equation}
By~\eqref{eq:estcoef}, we have that
\begin{align*}
\rho^k R |a_{k+1}^\pm - a_k^\pm| + |b_{k+1}- b_k| & = \rho^k R |a| + |b| \leq  C_1 C_0 d_k \rho^k R.
\end{align*}
Using~\eqref{eq:lk1}, \eqref{eq:estw}, \eqref{eq:polyclose} (with $\rho r$ instead of $r$, which remains true), and~\eqref{eq:estv}, we conclude that
\begin{align*}
\sup_{\Omega_{\rho^{k+1}R}^\pm } |u^\pm -l_{k+1}^\pm| 
& \leq \sup_{\Omega_{\rho r}^\pm } |u^\pm -p_{k} - v| + \sup_{\Omega_{\rho r}^\pm  } |p_k -l_k^\pm| + \sup_{\Omega_{\rho r}^\pm  } |v - l |\\
& \leq \sup_{B_{\rho r}} |w| + r\rho \omega(r\rho) +  \sup_{B_{\rho r} } |v - l | \\
& \leq \frac{1}{3}C_0 d_{k+1}R \rho^{k+1} + d_{k+1} \rho^{k+1} R+ \frac{1}{3}C_0 d_{k+1} \rho^{k+1} R\\
& \leq C_0 d_{k+1} \rho^{k+1} R,
\end{align*}
by choosing $C_0$ larger such that $C_0 \geq 3$.
\end{proof}

Next, we show that Theorem~\ref{thm2} follows from Theorem~\ref{thm3}.

\begin{proof}[Proof of Theorem~\ref{thm2}]
Take $R$ to be any $0 < R\leq 1$ such that $\gradiniz(R) \leq 1$ and $\int_{0}^{R} \frac{\gradiniz(t)}{t} d t \leq 1$.
Let $d_k$ be given as in~\eqref{eq:dk} for $k\geq 1$.
First, we will prove that
\begin{equation} \label{eq:finite}
    \sum_{j=1}^\infty d_j <\infty.
\end{equation}
Indeed, by definition, we have that $d_0 = \|u\|_{L^{\infty}(B_R)}R^{-1} + |g(0)|$ and
\begin{align*}
0 < d_k \leq \omega(\rho^k R) + \rho^{1/2} d_{k-1} \quad \hbox{for}~k\geq 1,
\end{align*}
where $\omega$ is given by~\eqref{def:modulus}.
For any $k\geq 1$, it follows that 
\begin{align*}
\sum_{j=1}^k d_j \leq \sum_{j=1}^k \omega(\rho^j R) + \rho^{1/2}\sum_{j=0}^{k-1} d_{j}
\leq \sum_{j=1}^k \omega(\rho^j R) + \rho^{1/2} d_0 + \rho^{1/2} \sum_{j=1}^{k} d_j.
\end{align*}
Since $\omega$ is a Dini function, from the proof of Lemma~\ref{lem:summable} and using that $\rho \leq 1/2$, we have
\[
\sum_{k = 1}^{\infty} \omega(\rho^{j}R) \leq
\frac{1}{1-\rho} \int_{0}^{R} \frac{\omega(t)}{t} d t 
\leq 2 \int_{0}^{R} \frac{\omega(t)}{t} d t,
\]
and thus,
\begin{equation}
\label{eq:util}
\begin{split}
\sum_{j=1}^k d_j  &\leq \frac{1}{1-\rho^{1/2}} \Big( \sum_{j=1}^k \omega(\rho^j R) + \rho^{1/2} d_0\Big) \\
&\leq C \left(\int_{0}^{R} \frac{\omega(t)}{t} d t+ \frac{\|u\|_{L^{\infty}(B_R)}}{R} + |g(0)|\right)\\
& \leq C \left(\int_{0}^{R} \frac{\gdiniz(t)}{t} d t+\frac{\|u\|_{L^{\infty}(B_R)}}{R} + |g(0)|\right),
\end{split}
\end{equation}
for some constant $C = C(n) > 0$,
where in the last line we used that
\[
\int_0^R\frac{\omega(t)}{t} d t \leq |g(0)| + \int_0^R\frac{\gdiniz(t)}{t} d t,
\]
thanks to the normalization for $\gradiniz$.
Letting $k\to \infty$, we conclude~\eqref{eq:finite}.

From~\eqref{eq:cauchy} and~\eqref{eq:finite}, it follows that there are $a^\pm \in \Rn$ and $b\in \R$ such that
\[
a_k^\pm \to a^\pm \quad \hbox{and}\quad b_k \to b \qquad\hbox{as}~k\to \infty.
\]
Moreover, by Theorem~\ref{thm3} and \eqref{eq:util} we deduce $a^+-a^-=g(0)e_n$, and the estimates
\[
\begin{split}
|a^{\pm}| &\leq |a_0^{\pm}| + \sum_{k = 1}^{\infty} |a_{k}^{\pm} - a_{k-1}^{\pm}| \leq \frac{1}{2}|g(0)| + C_1C_0 \sum_{k = 1}^{\infty}d_{k-1} \\
& \leq C \left( \int_{0}^{R} \frac{\gdiniz(t)}{t} d t+ \frac{\|u\|_{L^{\infty}(B_R)}}{R} + |g(0)|\right),
\end{split}
\]
and
\[
\begin{split}
|b| &\leq |b_0| + \sum_{k = 1}^{\infty} |b_k - b_{k-1}| \leq C_1 C_0 \sum_{k =1}^{\infty} d_{k-1} \rho^{k-1} R \leq C_1 C_0 R\sum_{k = 1}^{\infty} d_{k-1}\\
& \leq C \left( \int_{0}^{R} \frac{\gdiniz(t)}{t} d t + \|u\|_{L^{\infty}(B_R)} + R|g(0)|\right),
\end{split}
\]
where $C = C(n) > 0$ is a dimensional constant.

Define 
\[
l^\pm(x)=a^\pm \cdot x + b.
\]
Given $0 < r \leq R$, there is $k \geq 0$ such that $\rho^{k+1} R < r \leq \rho^{k} R$.
By Theorem~\ref{thm3}, we see that
\begin{align*}
\sup_{\opm \cap B_r} |u^\pm - l^\pm| & \leq \sup_{\opm \cap B_{\rho^kR}} |u^\pm - l_k^\pm|  + \sup_{\opm \cap B_{\rho^kR}} |l_k^\pm - l^\pm| \\
& \leq C_0 d_k \rho^k R + \rho^k R |a_k^\pm - a^\pm|  +|b_k - b|\\
& \leq C_0 d_k r + \rho^k R \sum_{j = k}^{\infty} |a_{j+1}^\pm - a_{j}^\pm|  +\sum_{j = k}^{\infty} |b_{j+1} - b_{j}|\\ 
& \leq C_0 d_k r + 2C_1 C_0 r \sum_{j = k}^{\infty} d_{j} 
\\
& \leq 3 C_1 C_0 r \sum_{j = k}^{\infty} d_{j} =  r \sigma(r),
\end{align*}
with
\[
\sigma(r) = 3 C_1 C_0 \sum_{j = k}^{\infty} d_{j} \quad \text{ for } \rho^{k+1} R < r \leq \rho^{k} R.
\]
Here, recall that the constants $C_0$ and $C_1$ depend only on $n$.
Furthermore, using~\eqref{eq:finite}, we get $\sigma(r) \to 0$ as $r\to 0$ (implying $k\to\infty$). We conclude the desired result.
\end{proof}

\begin{rem} \label{rem:natural}
We point out that, in the proof of Theorem~\ref{thm2}, it is crucial to have
    $$
    \sum_{j=0}^\infty \omega(\rho^j R)<\infty,
    $$
    where $\omega(r)=\max\{|g(0)|\gradiniz(r), \gdiniz(r)\}$.
    By Lemma~\ref{lem:summable}, this is equivalent to $\omega$ being a Dini function.
    In view of the counterexamples from Section~\ref{sec:counterexamples},
    the pointwise $C^{1,\dini}$ and $C^{0,\dini}$ assumptions on $\Gamma$ and $g$ in our main theorem
    are natural conditions to have piecewise differentiability of solutions to~\eqref{eq:main}.
\end{rem}

Finally, we give the proof of our main theorem.

\begin{proof}[Proof of Theorem~\ref{thm1}]
By Theorem~\ref{thm:existence}, we have that
$$
\|u\|_{L^\infty(\Omega)} \leq C \|g\|_{L^\infty(\Gamma)},
$$
where $C>0$ depends only on $n$, $\Omega$, and $\Gamma$. It remains to prove that 
\begin{equation}\label{eq:estimate}
    |u(x_1)-u(x_0)|\leq C \|g\|_{C^{0,\dini}(\Gamma)} |x_1-x_0| \qquad \text{ for all } x_0,x_1 \in \overline{\Omega}.
\end{equation}
By a standard covering argument, and using the classical interior and boundary estimates for harmonic functions, it is enough to prove~\eqref{eq:estimate} for $x_0\in \Gamma$ and $x_1 \in \Omega\setminus \Gamma$ sufficiently close to $x_0$.

Since $\Gamma$ is a $C^{1,\dini}$ interface, we may choose a radius $r_0$, independent of $x_0$, and an appropriate system of coordinates, such that 
$$x_0=0 \quad \hbox{and} \quad \Gamma \cap B_{r_0}= \{ x=(x',x_n)\in B_{r_0} : x_n = \psi(x')\},$$
for some $\psi \in C^{1,\dini}(0)$. Moreover, taking $ r_0^{-1}u(r_0 x)$ in place of $u$ we may assume that $r_0=1$. 

By Theorem~\ref{thm2}, there exist linear polynomials $l^\pm(x)= a^\pm \cdot x + b$, 
a radius $0 < R \leq 1$ depending only on $\Gamma$, and a function $\sigma(r)$ depending only on $n$, $\|g\|_\infty$, $\omega_\Gamma$, and $\omega_g$, such that
\begin{equation} \label{eq:eq3}
    \sup_{\opm \cap B_r} |u^\pm - l^\pm| \leq r \sigma(r) \qquad \text{ for all } r \leq R,
\end{equation}
 with $\sigma(r) \to 0$ as $r \to 0$.
In particular, $u(0) = b$. Furthermore,
\[
|a^{+}| + |a^{-}| \leq C \left(\|u\|_{L^{\infty}(B_R)} R^{-1}  + [g]_{C^{0,\dini}(0)} \right) \leq C \|g\|_{C^{0,\dini}(\Gamma)}
\]
for some constant $C > 0$ depending only on $n$, $\Omega$, and $\Gamma$.

Let $x_1\in \Omega^\pm$ and set $r=|x_1|$. If $r$ is sufficiently small so that $r\leq R$ and $\sigma(r)\leq 1$, then by \eqref{eq:eq3}, we get that
$$
|u(x_1)-u(0)| \leq |u^\pm(x_1) - l^\pm(x_1)| + |a^\pm| |x_1|  \leq  r \sigma(r) + |a^\pm| |x_1| \leq C\|g\|_{C^{0,\dini}(\Gamma)}|x_1|.
$$
Therefore, we have seen that there is a radius $r > 0$ depending only on $n$, $\Omega$, $\Gamma$, and $g$ such that, for all $x_0 \in \Gamma$, we have
$$
|u(x_1)-u(x_0)| \leq C\|g\|_{C^{0,\dini}(\Gamma)}|x_1-x_0| \quad \text{ for all } x_1 \in B_r(x_0),
$$
which yields~\eqref{eq:estimate}.
\end{proof}

\begin{proof}[Proof of Corollary~\ref{cor:piece}]
The result follows
from Proposition~\ref{prop:patching}. To apply this proposition, we have to check the
prerequisites (i), (ii), and (iii). Notice that (i) follows from interior estimates
for harmonic functions, (ii) holds because of Theorem~\ref{thm2}, and the uniform Dini condition. Furthermore, (iii) is satisfied due to the fact that $u - l_y$
satisfies (distributionally) the same equation as $u$ (and hence, our obtained
estimates also apply to $u - l_y$). The details are left to the reader.
\end{proof}

\appendix

\section{$C^1$ regularity up to the boundary}
\label{sec:c1}

The following abstract result provides sufficient conditions for a differentiable function defined on a $C^1$ bounded domain to be $C^1$ up to the boundary.
Analogous results have been obtained for $C^\alpha$, $C^{1,\alpha}$, and $C^{2,\alpha}$ spaces; see for instance \cite[Propositions~2.3 and 2.4]{MS} and \cite[Proposition~6.2]{CSCS}.
The general idea is to obtain global regularity results by patching local interior and boundary estimates.

\begin{prop} \label{prop:patching}
    Let $\Omega \subset \Rn$ be a bounded $C^{1}$~domain. 
    For $x\in\Omega$, let $d_x=\dist(x,\partial\Omega)$. 
    Let $u$ be a differentiable function defined on $\overline{\Omega}$ satisfying the following properties:
\begin{enumerate}[label=\emph{(\roman*)}]
\item \label{def:int} (Interior estimates).~There exists a constant $A>0$ such that for any $x\in\Omega$, there exists a linear polynomial $l_x$, with
\[
|l_x(x)| + d_x |\nabla l_x| \leq A \|u\|_{L^{\infty}(B_{d_x}(x))},
\]
and
\[
\sup_{B_{\rho}(x)} |u - l_x| \leq A \|u\|_{L^{\infty}(B_{d_x}(x))} \frac{\rho^2}{d_x^2} \quad \text{ for all } \rho \leq d_x/2.
\]

\item (Boundary estimates).~There exist 
a nondecreasing function $\sigma$, with $\sigma(\rho)\to 0$ as $\rho\to 0$, and a radius $R>0$, such that for any $y\in\partial\Omega$, there is a linear polynomial $l_y$ such that
\[
\sup_{B_{\rho}(y) \cap {\Omega}} |u - l_y| \leq  \rho \sigma(\rho) \quad \hbox{ for all } \rho \leq R.
\]

\item (Invariance property).~For any $y\in\partial\Omega$, the function $u-l_y$ also satisfies the interior estimates given in~\ref{def:int}.\medskip
\end{enumerate} 

Then
\[
\sup_{B_{\rho}(x) \cap \Omega} |u-l_x| \leq C \rho E(x, \rho) \qquad \hbox{for all } x \in \Omega \text{ and } \rho\leq R,
\]
where $C$ depends only on $A$,
and the error $E(x, \rho)$ is given by 
\[
E(x, \rho) = \max\{\sigma(2 d_x), \sigma(4\rho)\}.
\]
As a consequence,
$u\in C^1(\overline\Omega)$.
\end{prop}

\begin{proof}
Fix $x \in \Omega$ and let $y \in \partial \Omega$ be such that $|x- y| = d_x$. Let $l_x$ and $l_y$ be the linear polynomials given in (i) and (ii), respectively. By (iii), there is a linear polynomial $\widetilde{l}_x$ 
such that the interior estimates hold for $u-l_y$ in place of $u$. 
We want  to control
\[
\sup_{B_{\rho}(x) \cap \Omega} |u - l_y - \widetilde{l}_x|.
\]
We distinguish two cases:
\medskip

\noindent
\textbf{Case 1:} $0<\rho \leq d_x/2$.
By interior estimates for $u - l_y$, we have
\[
\begin{split}
\sup_{B_{\rho}(x) \cap \Omega}|u - l_y-\widetilde{l}_x| & \leq A \|u - l_y\|_{L^{\infty}(B_{d_x}(x))} \frac{\rho^2}{d_x^2}
\leq 2 A \frac{\sigma(2 d_x) \rho^2 }{d_x},
\end{split}
\]
where in the last line, we used the boundary estimates together with $B_{d_x}(x) \subset B_{2d_x}(y)$.

\medskip
\noindent
\textbf{Case 2:} $d_x/2 < \rho \leq R$.
By interior estimates for $u-l_y$ and boundary estimates for $u$, we have
\[
|\widetilde{l}_x(x)| + d_x |\nabla \widetilde{l}_x| \leq A \|u-l_y\|_{L^{\infty}(B_{d_x}(x))} \leq 2 A d_x \sigma(2 d_x),
\]
and by boundary estimates again,
\[
\begin{split}
\sup_{B_{\rho}(x) \cap \Omega} |u - l_y - \widetilde{l}_x| 
&\leq \sup_{B_{\rho}(x) \cap \Omega}|u - l_y| + |\widetilde{l}_x(x)| + |\nabla \widetilde{l}_x| \rho\\
& \leq 2  d_x \sigma(2 d_x) + 2 A  (d_x+ \rho) \sigma(2 d_x)\\
& \leq 4  \rho \sigma(4 \rho) + 6 A  \rho \sigma(4 \rho)\\
& \leq 2(2 + 3 A )  \rho \sigma(4 \rho).
\end{split}
\]

It follows that $l_x = \widetilde{l}_x+l_y$, where $l_x(z) = u(x) + \nabla u(x) \cdot (z-x)$, and
\[
\sup_{B_{\rho}(x) \cap \Omega} |u-l_x| \leq C
\begin{cases}
\frac{\sigma(2d_x) \rho^2}{d_x} & \text{ if } 0< \rho \leq d_x/2,\\
\rho \sigma(4\rho) & \text{ if } d_x/2< \rho \leq R,
\end{cases}
\]
for some constant $C$ depending only on $A$. 
In particular, we have the bound
\[
\sup_{B_{\rho}(x) \cap \Omega} |u-l_x| \leq C \rho 
E(x, \rho) \qquad \hbox{for all } \rho\leq R,
\]
proving the first claim. \medskip

To show that $u \in C^1(\overline{\Omega})$ from the previous estimate, 
we modify a trick that we learned from the recent book~\cite[Appendix A, Proof of (H4)]{Xavis}.
Let $x \in \Omega$ and $y\in \partial \Omega$, and write $\rho := |x-y|$.
Let $z \in \Omega$ be such that 
\[
\frac{\rho}{4} \leq |z-x| \leq 2 \rho \quad \text{ and } \quad \frac{\rho}{4} \leq |z-y| \leq 2 \rho.
\]
Then the following first order Taylor expansions hold as $\rho\to0$:
\[
\begin{split}
u(z) & = u(x) + \nabla u(x)\cdot (z-x) + \rho \, O(E(x, 2\rho))\\
&= u(y) + \nabla u(y) \cdot (z-y) + \rho \, O(E(y, 2\rho))\\
&= u(x) + \nabla u(x)\cdot (y-x) + \rho \, O(E(x, \rho)) + \nabla u(y) \cdot (z-y) + \rho \, O(E(y, 2\rho)),
\end{split}
\]
where the big $O$ notation, $f(x)=O(g(x))$ as $x\to x_0$, means that there is some constant $C>0$ such that $|f(x)|\leq C g(x)$ for all $x$ sufficiently close to $x_0$.
Hence,
\[
(\nabla u(x) - \nabla u(y)) \cdot \frac{(z-y)}{|z-y|} = O(E(y, 2\rho)) + O(E(x, 2\rho)).
\]

Let $\nu_y$ denote the inner unit normal to $\Omega$ at $y \in \partial \Omega$.
By $C^1$ regularity of $\partial \Omega$, for any direction $e \in \partial B_1$ with $e \cdot \nu_y > 0$, we can find $z \in \Omega$ as above such that $\frac{z-y}{|z-y|} = e$, provided $\rho$ is taken sufficiently small.
It follows that
\[
\lim_{x \to y} |\partial_e u(x) - \partial_e u(y)| \leq C \lim_{\rho = |x-y| \to 0} (E(y, 2\rho) + E(x, 2\rho)),
\]
where we write $\partial_e u(x) = \nabla u(x) \cdot e$.
Since $y \in \partial \Omega$, we have $d_y = 0$ and hence
\[
E(y, 2\rho) = \sigma(8 \rho) \to 0 \text{ as } \rho \to 0.
\]
Moreover, 
we also have that $d_x \to 0$ as $\rho = |x-y| \to 0$, which yields
\[
\lim_{\rho = |x-y| \to 0} E(x, 2\rho) \leq \max\left\{ \lim_{x \to y}\sigma(2d_x), \lim_{\rho \to 0}\sigma(8 \rho)\right\} = 0.
\]
Therefore, we conclude that 
\[
\lim_{x \to y} |\partial_e u(x) - \partial_e u(y)| = 0
\]
for all $e \in \partial B_1$ with $e \cdot \nu_y > 0$.
It remains to prove that the same limit continues to hold for directions $e \in \partial B_1$ with $e \cdot \nu_y = 0$.
For this, writing $e$ as $e = \nu_y - (\nu_y-e)$, since $(\nu_y -e) \cdot \nu_y = 1 > 0$, by the above we deduce
\[
\lim_{x\to y}|\partial_e u(x) - \partial_e u(y)| \leq \lim_{x\to y}|\partial_{\nu_y} u(x) - \partial_{\nu_y} u(y)| 
+\lim_{x\to y}|\partial_{\nu_y-e} u(x) - \partial_{\nu_y-e} u(y)| = 0.
\]
It follows that $u$ is $C^1$ up to the boundary.
\end{proof}

\section{Pointwise Dini vs. uniformly Dini}
\label{sec:pointwise}

As mentioned in Remark~\ref{rem:dini}, a function that is  $C^{0,\dini}$ for all points is not necessarily $C^{0,\dini}$ uniformly (see Definition~\ref{def:dinispaces}). The following lemma proves this fact.

\begin{lem} \label{lem:dini}
There exists a continuous function $g:[-1,1]\to \R$ such that $g\in C^{0,\dini}(x_0)$ for all $x_0 \in [-1,1]$, but $g\notin C^{0,\dini}([-1,1])$.
\end{lem}
\begin{proof}
For $j\geq 1$, let  
\[
a_j=\frac{\log 2}{\log(j+1)}  \quad \hbox{and} \quad h_j=\frac{1}{(\log\log(j+1))^2}.
\]
Note that $\{a_j\}_{j=1}^\infty$ is a monotone decreasing sequence such that $a_1 = 1$ and $a_j \to 0$ as $j\to \infty$. 
Hence, $(0,1]= \bigcup_{j=1}^\infty I_j,$ where $I_j=[a_{j+1},a_{j}]$ and ${\rm int}(I_j)\cap {\rm int}(I_k)=\emptyset$ for any $k\neq j$.
Let 
$$b_j= \frac{a_j+a_{j+1}}{2} \quad 
\quad \hbox{and} \quad m_j=\frac{2h_j}{|I_j|},$$
where $|I_j|=a_{j}-a_{j+1}$.
We define a continuous function $g \colon [0,1]\to \R$ by
 \[
 g(x) = 
 \sum_{j=1}^\infty 
 m_j(x-a_{j+1}) \chi_{(a_{j+1}, b_j]} + m_j(a_j - x) \chi_{(b_j, a_j]} \quad \hbox{for}~x\in[0,1].
 \]
Notice that $g(0) = 0$.
Geometrically, the interval $I_j$ is the base of an isosceles triangle with height $h_j\to 0$ as $j\to \infty$, and the graph of $g$ over $I_j$ is the union of the two remaining edges that form the triangle.
Taking the even reflection, we obtain a continuous function $g \colon [-1,1] \to \R$.

Given $x_0\in[-1,1]$ and $r\leq 1$, let 
\[
\omega_{g,x_0}(r)= \sup_{x\in (x_0-r, x_0+r)\cap [-1,1]} |g(x)-g(x_0)|.
\]
Since $\sup_{[-1,1]}|g|= h_1 \leq 1/(\log\log2)^2$, we have that $\omega_{g,x_0}(r)\leq 2/(\log\log2)^2$ for all $r\leq 1$.
For $x_0 \in (0,1]$ (the case $x_0 \in [-1,0)$ is analogous),
there exists $j\in \N$ such that $a_{j+1} < x_0 \leq a_{j}$.
If $x_0\neq a_j$, then for $\delta< \min \big(|x_0-a_{j}|, |x_0-a_{j+1}| \big)$ we have
\[
\omega_{g,x_0}(r) = \sup_{x\in (x_0-r, x_0+r)} |m_j(x-x_0)|= m_j r \qquad \hbox{for all}~r<\delta,
\]
and thus,
\[
\int_0^\delta \frac{\omega_{g,x_0}(r)}{r} \leq m_j \delta < \infty.
\]
Similarly, if $x_0 = a_j$, the above computation holds taking $\delta<|I_j|/4$, and noting that $\omega_{g,x_0}(r)\leq m_{j-1}r$ for all $r<\delta$.
It follows that $g \in C^{0,\dini}(x_0)$ for all $x_0  \in [-1,1] \setminus \{0\}$.

It remains to show the case $x_0 = 0$. 
Since $g(0) = 0$, we have that
\[
\omega_{g,0}(r) = \sup_{x \in (- r, r)}|g(x)| \qquad\hbox{for all}~r\leq 1.
\]
Moreover, $\sup_{I_j \cup -I_j}|g| = h_j$, and hence
\[
\int_0^1 \frac{\omega_{g,0}(r)}{r} \, dr 
= \sum_{j=1}^\infty \int_{a_{j+1}}^{a_j}\frac{\omega_{g,0}(r)}{r} \, dr \leq  \sum_{j=1}^\infty \int_{a_{j+1}}^{a_j}\frac{\sup_{I_j} |g| }{r}\, dr
=\sum_{j=1}^\infty h_j \log \left(\frac{a_j}{a_{j+1}}\right).
\]
We will see that the infinite sum converges.
Recall that $\log(1+x)\sim x$ as $x\to 0$ 
(here, we write $f(x)\sim h(x)$ as $x \to x_0$ if $f(x) / h(x) \to 1$ as $x\to x_0$).
We observe that
\[
\log\left(\frac{a_j}{a_{j+1}}\right) 
= \log \left( \frac{\log(j+2)}{\log(j+1)} \right) 
\sim \frac{\log(j+2)}{\log(j+1)} -1 = \frac{\log\big(1+\frac{1}{j+1}\big)}{\log(j+1)} 
\sim \frac{1}{j \log j} \quad \hbox{as}~j\to \infty.
\]
Hence,
\[
 h_j \log \left(\frac{a_j}{a_{j+1}}\right) \sim  \frac{1}{ (\log\log j)^2}  \frac{1}{j \log j} \quad \hbox{as}~j\to \infty.
\]
Furthermore, we have 
$$
\int_{e^e}^\infty \frac{1}{(\log\log x)^2} \frac{1}{x\log x}\, dx 
= \int_{1}^\infty \frac{1}{y^2}\, dy < \infty,
$$
where we used the change of variables $y = \log\log x$.
By the limit comparison test and the integral test, it follows that
$$\sum_{j=1}^\infty h_j \log \left(\frac{a_j}{a_{j+1}}\right)<\infty,$$
and we conclude $g \in C^{0,\dini}(0)$.
Therefore, we have proved that $g\in C^{0,\dini}(x_0)$ for all $x_0\in [-1,1]$. 

Now, for $r\leq 1$, consider
$$
\omega_g(r) = \sup_{x_0\in [-1,1]} \omega_{g,x_0}(r).
$$
Note that $\omega_g$ is well defined since $\omega_{g,x_0}$ is uniformly bounded. 
Fix $0<R\leq 1$. Since $|I_j|\to 0$ as $j\to\infty$, there exists $j_0\geq 1$ such that $|I_j|< 4R$ for all $j \geq j_0$.
Then, for all $j\geq j_0$, we have 
\begin{align} \label{eq:eq1}
\int_0^R \frac{\omega_g(r)}{r}\, dr 
&\geq \int_{|I_j|/4}^R \frac{\omega_{g,b_j}(r)}{r}\, dr
\geq \frac{m_j |I_j| }{4} \int_{|I_j|/4}^R \frac{1}{r}\, dr
= 2 h_j \log R -2 h_j \log |I_j|.
\end{align}
It suffices to show that $- h_j  \log |I_j| \to \infty$ as $j\to \infty$. We have 
$$
|I_j| = \frac{\log 2}{\log(j+2)} \frac{\log\big(1+\frac{1}{j+1}\big)}{\log(j+1)}\sim\frac{\log 2}{j (\log j)^2}
\quad \hbox{as}~ j\to\infty,
$$
and thus,
\begin{equation} \label{eq:eq2}
- h_j \log |I_j| \sim  \frac{1}{(\log \log j)^2} \log\left( \frac{j (\log j)^2}{\log 2}\right) \sim \frac{\log j}{(\log \log j)^2} \to \infty \quad \hbox{as}~j\to \infty.
\end{equation}
Combining \eqref{eq:eq1} and \eqref{eq:eq2}, letting $j \to \infty$, we see that
$$
\int_0^R \frac{\omega_g(r)}{r}\, dr =\infty \qquad \hbox{for any}~0 < R\leq 1.
$$
Therefore, $\omega_g$ is not a Dini function, and we conclude that $g\notin C^{0,\dini}([-1,1])$.
\end{proof}

\subsection*{Acknowledgments}
The authors wish to thank the referee for their careful reading of the manuscript, as well as for their useful suggestions that greatly improved the presentation of this paper.
We would also like to thank Xavier Ros-Oton for a useful hint for the proof of Proposition~\ref{prop:patching}.


\begin{thebibliography}{10}

\bibitem{AC} H. W.  Alt and L. A. Caffarelli.
Existence and regularity for a minimum problem with free boundary.
\textit{J. Reine Angew. Math.} 325 (1981), 105--144.

\bibitem{ACF} H. W.  Alt, L. A. Caffarelli, and A. Friedman.
Variational problems with two phases and their free boundaries.
\textit{Trans. Amer. Math. Soc.} 282 (1984), 431--461.

\bibitem{ACS} I. Athanasopoulos, L. A. Caffarelli, and S. Salsa.
The free boundary in an inverse conductivity problem.
\textit{J. Reine Angew. Math.} 534 (2001), 1--31.

\bibitem{C1} L. A. Caffarelli.
Elliptic second order equations.
\textit{Rend. Sem. Mat. Fis. Milano} 58 (1988), 253--284.

\bibitem{C2} L. A. Caffarelli.
Interior a priori estimates for solutions of fully nonlinear equations.
\textit{Ann. of Math. (2)} 130 (1989), no.~1, 189--213.

\bibitem{CSCS} L.~A.~Caffarelli, M.~Soria-Carro, and P.~R.~Stinga.
{Regularity for $C^{1,\alpha}$ interface transmission problems}.
\textit{Arch. Ration. Mech. Anal.}
{240} (2021), 265--294.

\bibitem{D} H. Dong.
A simple proof of regularity for {$C^{1,\alpha}$} interface transmission problems. \textit{Ann. Appl. Math.} 37 (2021), no.~1, 22--30.

\bibitem{DEK} H. Dong, L. Escauriaza, and S. Kim.
On  C1,  C2, and weak type-(1,1)  estimates for linear elliptic operators: part~II.
\textit{Math. Ann.} 370 (2018), no. 1-2, 447--489.

\bibitem{DLK} H. Dong, J. Lee, and S. Kim.
On conormal and oblique derivative problem for elliptic equations with Dini mean oscillation coefficients. \textit{Indiana Univ. Math. J.} 69 (2020), no.~6, 1815--1853.

\bibitem{DM} F.  Duzaar and G. Mingione.
Gradient estimates via linear and nonlinear potentials.
\textit{J. Funct. Anal.} 259 (2010), no. 11, 2961--2998.

\bibitem{Xavis} X. Fern\'andez-Real and X. Ros-Oton.
\emph{Regularity theory for elliptic {PDE}}.
Zur. Lect. Adv. Math. 28.
EMS Press, Berlin, 2022.

\bibitem{HLW} Y. Huang, D. Li, and L. Wang.
Boundary behavior of solutions of elliptic equations in nondivergence form. 
\textit{Manuscripta Math.} 143 (2014), 525--541.

\bibitem{HLZ} Y. Huang, D. Li, and K. Zhang.
Pointwise boundary differentiability of solutions of elliptic equations.
\textit{Manuscripta Math.} 151 (2016), no.~3-4, 469--476.

\bibitem{K1} T. Kilpel\"{a}inen. H\"{o}lder continuity of solutions to quasilinear elliptic equations involving measures.
\textit{Potential Anal.} 3 (1994), no.~3, 265--272.

\bibitem{KM} T. Kilpel\"{a}inen and J. Mal\'{y}.
Degenerate elliptic equations with measure data and nonlinear potentials.
\textit{Ann. Scuola Norm. Sup. Pisa Cl. Sci. (4)} 19 (1992), no.~4, 591--613.
  
\bibitem{KS} S. Kim and G. Sakellaris.
Green's function for second order elliptic equations with singular lower order coefficients.
\textit{Comm. Partial Differential Equations} 44 (2019), no.~3, 228--270.
  
\bibitem{LXZ}  Y. Lian, W. Xu, and K. Zhang.
Boundary Lipschitz regularity and the Hopf lemma on Reifenberg domains for fully nonlinear elliptic equations. 
\textit{Manuscripta Math.} 166 (2021),  no.~3-4, 343--357.

\bibitem{Lieberman} G. M. Lieberman.
Sharp forms of estimates for subsolutions and supersolutions of quasilinear elliptic equations involving measures.
\textit{Comm. Partial Differential Equations} 18 (1993), no. 7-8, 1191--1212.

\bibitem{LSW} W. Littman, G. Stampacchia, and H. F. Weinberger. 
Regular points for elliptic equations with discontinuous coefficients. 
\textit{Ann. Scuola Norm. Sup. Pisa Cl. Sci.}
{17} (1963), no.~3, 43--77.

\bibitem{MW} F. Ma and L. Wang.
Boundary first order derivative estimates for fully nonlinear elliptic equations.
\textit{J. Differential Equations} 252 (2012), no.~2, 988--1002.

\bibitem{MS} E. Milakis and L. E. Silvestre.     
Regularity for fully nonlinear elliptic equations with Neumann boundary data.
\textit{Comm. Partial Differential Equations} 31 (2006), no. 7-9, 1227--1252.

\bibitem{Mingione} G. Mingione.
Gradient potential estimates.
\textit{J. Eur. Math. Soc.} 13 (2011), 459--486.  	

\bibitem{Muller1} M. M\"{u}ller.
The Poisson equation involving surface measures.
\textit{Comm. Partial Differential Equations} 47 (2022), no. 5, 948--988.

\bibitem{Muller2} M. M\"{u}ller.
Erratum: The Poisson equation involving surface measures.
\textit{Comm. Partial Differential Equations} 48 (2023), no.~2, 350--353.

\bibitem{Muller3} M. M\"{u}ller.
On elliptic equations involving surface measures.
To appear in \textit{Ann. Sc. Norm. Super. Pisa Cl. Sci.}
\url{https://doi.org/10.2422/2036-2145.202303_012}

\bibitem{Ponce} A. C. Ponce.
\textit{Elliptic PDEs, measures and capacities.}
EMS Tracts Math., 23.
European Mathematical Society (EMS), Z\"{u}rich, 2016.

\bibitem{RZ} J.-M. Rakotoson and W. Ziemer.
Local behavior of solutions of quasilinear elliptic equations with general structure.
\textit{Trans. Amer. Math. Soc.} 319 (1990), no.~2, 747--764.

\bibitem{WLZ} D. Wu,  Y. Lian, and K. Zhang.
Pointwise boundary differentiability for fully nonlinear elliptic equations.
\textit{Israel J. Math.} 258 (2023), no.~1, 375--401.

\end{thebibliography}
\end{document}